\documentclass[11pt,reqno]{amsart}

\usepackage[colorlinks,linkcolor={blue},
citecolor={blue},urlcolor={blue}]{hyperref}

\usepackage{mathrsfs}
\usepackage{bbm}
\usepackage{amsfonts}

\usepackage{}
\usepackage{pifont}
\usepackage[shortlabels]{enumitem}
\usepackage{empheq}
\usepackage{amsfonts}
\usepackage{amsmath, amsrefs, amsthm, amssymb}
\usepackage[active]{srcltx}		
\usepackage{pdfsync}					
\usepackage{graphicx}
\usepackage [latin1]{inputenc}
\usepackage{bbm}
\usepackage{amsfonts}
\usepackage{amsthm}
\usepackage{graphicx}
\usepackage{framed}
\usepackage{multicol,multienum}
\usepackage{graphicx} 
\usepackage{booktabs} 
\usepackage{mathrsfs}
\usepackage{tcolorbox}
\usepackage{color}
\usepackage{xcolor}

\newtheorem{theorem}{Theorem}[section]
\newtheorem{proposition}[theorem]{Proposition}
\newtheorem{lemma}[theorem]{Lemma}
\newtheorem{corollary}[theorem]{Corollary}

\theoremstyle{definition}

\newtheorem{remark}[theorem]{Remark}

\numberwithin{equation}{section}

\begin{document}

\title [Chemotaxis-Navier-Stokes system in $\mathbb{R}^2$]{Existence and uniqueness of global large-data solutions for the Chemotaxis-Navier-Stokes system in $\mathbb{R}^2$}

\author{Fan Xu}
\address{School of Mathematics and Statistics, Hubei Key Laboratory of Engineering Modeling  and Scientific Computing, Huazhong University of Science and Technology,  Wuhan 430074, Hubei, P.R. China.}
\email{d202280019@hust.edu.cn (F. Xu)}

\author{Bin Liu}
\address{School of Mathematics and Statistics, Hubei Key Laboratory of Engineering Modeling  and Scientific Computing, Huazhong University of Science and Technology,  Wuhan 430074, Hubei, P.R. China.}
\email{binliu@mail.hust.edu.cn (B. Liu)}


\date{\today}

\begin{abstract}
This work investigates the Cauchy problem for the classical Chemotaxis-Navier-Stokes (CNS) system in $\mathbb{R}^2$. We establish the global existence and uniqueness of strong, classical, and arbitrarily smooth solutions under large initial data, which has not been addressed in the existing literature. The key idea is to first derive an entropy-energy estimate for initial data with low regularity, by leveraging the intrinsic entropy structure of the system. Building on this foundation, we then obtain higher-order energy estimates for smoother initial data via a bootstrap argument, in which the parabolic nature of the CNS system plays a crucial role in the iterative control of regularity.
\end{abstract}

\maketitle
\section{Introduction}
\subsection{Statement of the problem}
Chemotaxis is a fundamental biological phenomenon observed in a wide variety of organisms, ranging from bacteria and unicellular eukaryotes to complex multicellular systems. It enables these organisms to detect and respond to spatial gradients of chemical concentrations in their environment~\cite{keller1970,keller1971}. Essentially, chemotaxis functions as a directional sensing mechanism, allowing organisms to migrate toward favorable conditions or to avoid adverse stimuli. In 2005, Tuval et al.~\cite{tuval2005} proposed the Chemotaxis-Navier-Stokes (CNS) model to capture the intricate dynamics of bacterial suspensions within a fluid medium. The CNS system mathematically describes the coupled interactions between bacterial density, chemical concentration, and fluid flow, incorporating the effects of bacterial buoyancy on the fluid velocity field.

In the study of the CNS system, considerable attention has been devoted to analyzing the influence of viscous fluid effects on chemotactic behavior within the framework of partial differential equations. Existing results primarily concern the global existence of solutions~\cite{chae2013existence,ding2022generalized,duan2010global,jeong2022well,liu2011coupled,lorz2010coupled,winkler2012global,winkler2017far,winkler2022does,zhang2014global} and their long-time asymptotic behavior~\cite{chae2014global,17di2010chemotaxis,duan2017global}, with the goal of deepening the understanding of the complex interplay between chemotaxis and fluid dynamics.

In particular, in a two-dimensional smooth bounded domain, Winkler~\cite{winkler2012global} established the existence of a unique global classical solution by employing Kato's method combined with semigroup smoothing estimates. In the whole space $\mathbb{R}^2$, Duan, Lorz, Markowich and Liu~\cite{duan2010global,liu2011coupled} proved the global existence of weak solutions, and later, Zhang and Zheng~\cite{zhang2014global} further established both the existence and uniqueness of global weak solutions under comparatively mild assumptions. Motivated by these results in both bounded and unbounded settings, a natural mathematical question arises: \emph{Under large initial data conditions, does the CNS system admit a unique global strong, classical, and smooth solution in $\mathbb{R}^2$?} The primary objective of this study is to provide an affirmative answer to this question. Specifically, we aim to establish the global well-posedness of the following classical CNS system:
\begin{equation}\label{CNS0}
\left\{
\begin{aligned}
&n_t+u\cdot\nabla n=\Delta n-\nabla\cdot(n\nabla c),&& \textrm{in}~ \mathbb{R}_+\times \mathbb{R}^2,\\
&c_t+u\cdot\nabla c=\Delta c-nc,&&\textrm{in}~ \mathbb{R}_+\times\mathbb{R}^2,\\
&u_t+(u\cdot\nabla)u=\Delta u+\nabla P+n\nabla\phi,&&\textrm{in}~\mathbb{R}_+\times\mathbb{R}^2,\\
&\nabla\cdot u=0,&&\textrm{in}~ \mathbb{R}_+\times\mathbb{R}^2,\\
&n|_{t=0}=n_0,~c|_{t=0}=c_0,~u|_{t=0}=u_0,&&\textrm{in}~\mathbb{R}^2,
\end{aligned}
\right.
\end{equation}
where $\mathbb{R}_+:=[0,\infty)$, and the unknown functions $n(t,x)$, $c(t,x)$, and $u(t,x)$ represent the cell density, the chemical concentration, and the fluid velocity field, respectively. The scalar function $P(t,x)$ denotes the pressure, and $\phi(x)$ stands for the gravitational potential.

\subsection{Notations}\label{nnnn}
In this paper, we adopt the following conventions: The inequality $A \lesssim_{a,b,\cdots} B$ means that there exists a positive constant $C$ depending only on $a, b, \cdots$ such that $A \leq C B$, while $A \asymp_{a,b,\cdots} B$ indicates that there exist two positive constants $C \leq C'$ depending only on $a, b, \cdots$ such that $C B \leq A \leq C' B$.

Let $X$ and $Y$ be two Banach spaces. We denote by $\mathcal{L}(X; Y)$ the space of all bounded linear operators from $X$ to $Y$. The symbol $\langle \cdot, \cdot \rangle_{X', X}$ stands for the standard duality pairing, where $X' := \mathcal{L}(X; \mathbb{R})$ is the dual space of $X$. In particular, if $X$ is a Hilbert space, then the symbol $(\cdot, \cdot)_X$ denotes the scalar product.

Let $\mathcal{A}=(\mathcal{A},\Sigma,\mu)$ be a measure space, where $\Sigma$ is a $\sigma$-algebra on $\mathcal{A}$, and $\mu$ is a non-negative measure. Let $\mathcal{B}$ be a Banach space over $\mathbb{R}$. For $p\in[1,\infty)$: The space $L^p(\mathcal{A};\mathcal{B})$ is defined as the set of all (equivalence classes of) strongly measurable functions $f:\mathcal{A}\mapsto \mathcal{B}$ such that $\|f\|_{\mathcal{B}}^p$ is integrable over $\mathcal{A}$ with respect to $\mu$. For $p=\infty$: The space $L^\infty(\mathcal{A};\mathcal{B})$ is defined as the set
of all (equivalence classes of) strongly measurable functions $f:\mathcal{A}\mapsto \mathcal{B}$ that are essentially bounded on $\mathcal{A}$.

Let $\mathcal{V}:=\{f\in \mathcal{C}_0^{\infty}(\mathbb{R}^2;\mathbb{R}^2):\nabla\cdot f=0\}$, where $\mathcal{C}_0^{\infty}(\mathbb{R}^2;\mathbb{R}^2)$ denotes the space of all $\mathbb{R}^2$-valued functions of class $\mathcal{C}^{\infty}$ with compact supports. Let $\Lambda^s:=(I-\Delta)^{\frac{s}{2}}$ denote the Bessel operator \cite{bahouri2011fourier}. Let $H^s(\mathbb{R}^2)$ be a Hilbert space endowed with the norm
\begin{equation*}
\begin{split}
\|f\|_{H^s}=\|\Lambda^sf\|_{L^2}=\left[\int_{\mathbb{R}^2}(1+|\xi|^2)^s|\hat{f}(\xi)|^2\textrm{d}\xi\right]^{\frac{1}{2}},
 \end{split}
 \end{equation*}
where $\hat{f}$ denotes the Fourier transform of a tempered distribution $f$. Let
\begin{equation*}
\begin{split}
\mathbb{H}^s:=\textrm{the closure of}~\mathcal{V}~ \textrm{in}~H^s(\mathbb{R}^2;\mathbb{R}^2).
\end{split}
\end{equation*}

Let $Q_w$ denote the Hilbert space $Q$ endowed with the weak topology. Let $\{\mathcal{O}_d\}_{d\in\mathbb{N}}$ be a sequence of bounded open subsets of $\mathbb{R}^2$ with regular boundaries $\partial\mathcal{O}_d$ such that $\mathcal{O}_d\subset\mathcal{O}_{d+1}$ and $\bigcup_{d=1}^{\infty}\mathcal{O}_d=\mathbb{R}^2$.  Let $H^k(\mathcal{O}_d)$ denote the space of restrictions of functions defined on $\mathbb{R}^2$ to subsets $\mathcal{O}_d$, i.e.
$$H^k(\mathcal{O}_d):=\{f|_{\mathcal{O}_d}; f\in H^k(\mathbb{R}^2)\},\quad k\geq0.$$
We will use the following spaces in the subsequent content:
\begin{equation*} \begin{split}
\mathcal{C}([0,T];Q_w):=& \textrm{the~space~of~weakly~continuous~functions}~f:[0,T]\rightarrow Q\\
 & \textrm{with~the~weakest~topology~}\mathcal{T}_1\textrm{~such~that~for~all~}g\in Q \textrm{~the mappings}\\
 &\mathcal{C}([0,T];Q_w)\ni f\mapsto(f(\cdot),g)_{Q}\in \mathcal{C}([0,T];\mathbb{R})~\textrm{are~continuous}.\\
L^2_w(0,T;Q):=& \textrm{the~space}~L^2(0,T;Q)~\textrm{endowed~with~the~weak~topology~}\mathcal{T}_2.\\
L^2(0,T;H^k_{loc}(\mathbb{R}^2)):=& \textrm{the~space~of~measurable~functions}~f:[0,T]\rightarrow~H^k(\mathbb{R}^2)~\textrm{such~that~for~all}~d\in\mathbb{N}\\
 & p_{T,d}(f):=\|f\|_{L^2(0,T;H^k(\mathcal{O}_d))}:=\left(\int_0^T\|f(t)\|_{H^k(\mathcal{O}_d)}^2\mathrm{d}t\right)^{\frac{1}{2}}<\infty,\\
 & \textrm{with~the~topology}~\mathcal{T}_3\textrm{~generated~by~the~seminorms}~(p_{T,d})_{d\in\mathbb{N}}.
 \end{split} \end{equation*}

\subsection{Main results} \label{sec1.2}
Our main results are as follows.

\begin{theorem}\label{th1}
Suppose the initial data $(n_0,c_0,u_0)$ and the potential function $\phi$ satisfy
\begin{equation*}
\begin{split}
&n_0 \in L^1(\mathbb{R}^2)\cap H^1(\mathbb{R}^2),~n_0>0;\quad c_0 \in L^1(\mathbb{R}^2)\cap H^2(\mathbb{R}^2),~|\nabla \sqrt{c_0}|\in L^2(\mathbb{R}^2),~c_0>0;\\
&u_0 \in \mathbb{H}^1;\quad \nabla\phi\in L^{\infty}(\mathbb{R}^2).
 \end{split}
\end{equation*}
Then system \eqref{CNS0} admits a unique global strong solution such that for every $T>0$
\begin{equation*}
\begin{split}
&n \in  \mathcal{C}([0,T];H^1(\mathbb{R}^2))\cap L^2(0,T;H^2(\mathbb{R}^2));\quad c \in  \mathcal{C}([0,T];H^2(\mathbb{R}^2))\cap L^2(0,T;H^3(\mathbb{R}^2));\\
&u \in  \mathcal{C}([0,T]; \mathbb{H}^1)\cap L^2(0,T;\mathbb{H}^2).
 \end{split}
\end{equation*}
and for any $t \in (0,T]$, the integral identities
\begin{equation}\label{iden}
\begin{split}
(n(t),\varphi)_{L^2}&=(n_0,\varphi)_{L^2}+\int_0^t(\Delta n(r)-u(r)\cdot\nabla n(r)-\nabla\cdot (n(r)\nabla c(r)),\varphi)_{L^2}\mathrm{d}r,\\
(c(t),\varphi)_{L^2}&=(c_0,\varphi)_{L^2}+\int_0^t(\Delta c(r)-u(r)\cdot\nabla c(r)-n(r)c(r),\varphi)_{L^2}\mathrm{d}r,\\
(u(t),\psi)_{L^2}&=(u_0,\psi)_{L^2}+\int_0^t(\Delta u(r)-(u(r)\cdot\nabla)u(r)+n(r) \nabla\phi,\psi)_{L^2}\mathrm{d}r
 \end{split}
\end{equation}
hold for every $\varphi\in L^2(\mathbb{R}^2)$ and $\psi\in \mathbb{H}^0$.
\end{theorem}

\begin{theorem}\label{th2}
Let $s \geq 2$ be an integer. Suppose the initial data $(n_0, c_0, u_0)$ and the potential function $\phi$ satisfy
\begin{equation*}
\begin{split}
&n_0 \in L^1(\mathbb{R}^2)\cap H^{s}(\mathbb{R}^2),~n_0>0;\quad c_0 \in L^1(\mathbb{R}^2)\cap H^{s+1}(\mathbb{R}^2),~|\nabla \sqrt{c_0}|\in L^2(\mathbb{R}^2),~c_0>0;\\
&u_0 \in \mathbb{H}^{s};\quad\phi\in H^{s+1}(\mathbb{R}^2).
 \end{split}
\end{equation*}
Then system \eqref{CNS0} exists a unique global solution
\begin{equation*}
\begin{split}
&n\in\mathcal{C}([0,T];H^s(\mathbb{R}^2))\cap L^2(0,T;H^{s+1}(\mathbb{R}^2));\quad c\in\mathcal{C}([0,T];H^{s+1}(\mathbb{R}^2))\cap L^2(0,T;H^{s+2}(\mathbb{R}^2));\\
&u\in\mathcal{C}([0,T];\mathbb{H}^s)\cap L^2(0,T;\mathbb{H}^{s+1}).
\end{split}
\end{equation*}
In particular, if $s=4$, then system \eqref{CNS0} possesses a unique global classical solution
\begin{equation*}
\begin{split}
&n\in\mathcal{C}([0,T];\mathcal{C}^2(\mathbb{R}^2))\cap \mathcal{C}^1([0,T];\mathcal{C}(\mathbb{R}^2));\quad c\in\mathcal{C}([0,T];\mathcal{C}^3(\mathbb{R}^2))\cap \mathcal{C}^1([0,T];\mathcal{C}^2(\mathbb{R}^2));\\
&u\in\mathcal{C}([0,T];\mathcal{C}^2(\mathbb{R}^2;\mathbb{R}^2))\cap \mathcal{C}^1([0,T];\mathcal{C}(\mathbb{R}^2;\mathbb{R}^2)).
\end{split}
\end{equation*}
If $s=4+2m,~m\in\mathbb{N}^+$, then system \eqref{CNS0} admits a unique global smooth solution
\begin{equation*}
\begin{split}
&n\in\bigcap_{k=0}^{m+1}\mathcal{C}^k([0,T];\mathcal{C}^{2(m+1-k)}(\mathbb{R}^2));\quad c\in\bigcap_{k=0}^{m+1}\mathcal{C}^k([0,T];\mathcal{C}^{2(m+1-k)+1}(\mathbb{R}^2));\\
&u\in\bigcap_{k=0}^{m+1}\mathcal{C}^k([0,T];\mathcal{C}^{2(m+1-k)}(\mathbb{R}^2;\mathbb{R}^2)).
\end{split}
\end{equation*}
\end{theorem}
\begin{remark}
Theorems~\ref{th1} and~\ref{th2} demonstrate that the CNS system~\eqref{CNS0} admits a unique global strong, classical, and even smooth solution in $\mathbb{R}^2$ under suitably large initial data. These results refine and extend the global well-posedness theory for the CNS system established by Zhang and Zheng~\cite[Theorem~1]{zhang2014global} and Liu and Lorz~\cite[Theorem~2.1]{liu2011coupled}, which concerned only the existence and uniqueness of global weak solutions.
\end{remark}

\begin{remark}
The core idea behind the proof is as follows. Under initial data of low regularity, system~\eqref{CNS0} possesses a special entropy structure that enables the derivation of uniform low-order entropy-energy estimates (cf. Lemma~\ref{5lem}). As the regularity of the initial data improves, the parabolic nature of system~\eqref{CNS0} becomes dominant. In particular, the low-order entropy-energy estimate facilitates the derivation of the energy bound $\|n^{\epsilon}(t)\|_{H^1}^2 + \|c^{\epsilon}(t)\|_{H^2}^2 + \|u^{\epsilon}(t)\|_{H^1}^2$ (cf. Lemma~\ref{5pro}), which, combined with weak compactness arguments, guarantees the global existence of strong solutions. With further enhancement in the regularity of the initial data, the intrinsic parabolicity of the CNS system, together with the structure of the cross-diffusion term, allows us to establish uniform energy estimates for $\|n^{\epsilon}(t)\|_{H^s}^2 + \|c^{\epsilon}(t)\|_{H^{s+1}}^2 + \|u^{\epsilon}(t)\|_{H^s}^2$, for any $s\geq2$ (cf. Proposition~\ref{6pro}), which in turn yields the global existence of classical, and even arbitrarily smooth, solutions to system~\eqref{CNS0} (cf. Corollary~\ref{cor2}).
\end{remark}

\subsection{Structure of the paper}
Section~\ref{sec2} is devoted to establishing the existence and uniqueness of sufficiently smooth solutions to the regularized problem~\eqref{Mod-1}. The proofs of Theorems~\ref{th1} and~\ref{th2} are given in Sections~\ref{sec3} and~\ref{sec4}, respectively.

\section{Solvability of the modified systems}\label{sec2}
This section is devoted to the construction of an approximate system corresponding to \eqref{CNS0}. Let $\rho^\epsilon$ be a standard Friedrichs mollifier (cf.~\cite{majda2002vorticity}). We consider the following regularized system:
\begin{equation}\label{Mod-1}
\left\{
\begin{aligned}
&n^\epsilon_t + u^\epsilon \cdot \nabla n^\epsilon = \Delta n^\epsilon - \nabla \cdot \left(n^\epsilon \nabla (c^\epsilon * \rho^\epsilon) \right), \\
&c^\epsilon_t + u^\epsilon \cdot \nabla c^\epsilon = \Delta c^\epsilon - c^\epsilon (n^\epsilon * \rho^\epsilon), \\
&u^\epsilon_t + \mathbf{P}[(u^\epsilon \cdot \nabla) u^\epsilon] = - A u^\epsilon + \mathbf{P}[(n^\epsilon \nabla \phi) * \rho^\epsilon], \\
&\nabla \cdot u^\epsilon = 0, \\
&n^\epsilon|_{t=0} = n_0^\epsilon, \quad c^\epsilon|_{t=0} = c_0^\epsilon, \quad u^\epsilon|_{t=0} = u_0^\epsilon,
\end{aligned}
\right.
\end{equation}
where the initial data are smoothed by convolution with the mollifier:
\[
n_0^\epsilon = n_0 * \rho^\epsilon, \quad c_0^\epsilon = c_0 * \rho^\epsilon, \quad u_0^\epsilon = u_0 * \rho^\epsilon.
\]
Here, $A := -\mathbf{P} \Delta$ denotes the Stokes operator, where $\mathbf{P}$ is the Helmholtz-Leray projection from $L^2(\mathbb{R}^2;\mathbb{R}^2)$ onto the divergence-free subspace $\mathbb{H}^0$.

Let $J_k$ be the frequency truncation operator defined by
\[
\widehat{J_k f}(\xi) := \mathbf{1}_{\mathcal{C}_k}(\xi) \widehat{f}(\xi), \quad \mathcal{C}_k := \left\{\xi \in \mathbb{R}^2 \; ; \; k^{-1} \leq |\xi| \leq k \right\}, \quad k \in \mathbb{N}^+.
\]
It is well known that the operators $\mathbf{P}$, the Bessel potential operator $\Lambda^s$, and the frequency truncation operator $J_k$ commute pairwise. Moreover, the following equality holds:
$$
(\textbf{P}f,g)_{H^m}=(f,\textbf{P}g)_{H^m}=(f,g)_{H^m},\quad f\in H^m(\mathbb{R}^2;\mathbb{R}^2),~ g\in \mathbb{H}^m,~ m\in\mathbb{N}.
$$

We aim to prove that the regularized system \eqref{Mod-1} admits a unique sufficiently smooth solution. To this end, we consider the following second-level approximation:
\begin{equation}\label{mod-22}
\left\{
\begin{aligned}
&n^{k,\epsilon}_t + J_k(J_k u^{k,\epsilon} \cdot \nabla J_k n^{k,\epsilon}) = \Delta J_k^2 n^{k,\epsilon} - J_k( \nabla \cdot (J_k n^{k,\epsilon} \nabla (c^{k,\epsilon} * \rho^\epsilon)) ) ,\\
&c^{k,\epsilon}_t + J_k(J_k u^{k,\epsilon} \cdot \nabla J_k c^{k,\epsilon}) = \Delta J_k^2 c^{k,\epsilon} - J_k( J_k c^{k,\epsilon} (n^{k,\epsilon} * \rho^\epsilon)), \\
&u^{k,\epsilon}_t + J_k \mathbf{P} [ (J_k u^{k,\epsilon} \cdot \nabla) J_k u^{k,\epsilon} ] = - A J_k^2 u^{k,\epsilon} + J_k \mathbf{P} [ (n^{k,\epsilon} \nabla \phi) * \rho^\epsilon ], \\
&\nabla \cdot u^{k,\epsilon} = 0, \\
&n^{k,\epsilon}|_{t=0} = n_0^\epsilon, \quad c^{k,\epsilon}|_{t=0} =c_0^\epsilon, \quad u^{k,\epsilon}|_{t=0} = u_0^\epsilon.
\end{aligned}
\right.
\end{equation}
System \eqref{mod-22} can be expressed as an ODE in the work space $\textbf{H}^s(\mathbb{R}^2):=H^s(\mathbb{R}^2)\times H^s(\mathbb{R}^2) \times \mathbb{H}^s$:
\begin{equation}\label{Mod-22}
\left\{
\begin{aligned}
& \frac{\mathrm{d}}{\mathrm{d}t}  \textbf{y}^{k,\epsilon}= F ^{k,\epsilon} (\textbf{y}^{k,\epsilon}  ) ,\\
& \textbf{y}^{k,\epsilon}(0) =\textbf{y}^{\epsilon}_0= (n_0^\epsilon,c_0^\epsilon,u_0^\epsilon)^{\top},
\end{aligned}
\right.
\end{equation}
where $\textbf{y}^{k,\epsilon}:=(n^{k,\epsilon} , c^{k,\epsilon},
u^{k,\epsilon})^{\top}$ and $F^{k,\epsilon}(\cdot):=(F_1^{k,\epsilon} (\cdot),F_2^{k,\epsilon} (\cdot),F_3^{k,\epsilon}(\cdot))^{\top}$. The functionals $F_i^{k,\epsilon}(\cdot)$ are given by
\begin{equation*}
\begin{split}
 F_1^{k,\epsilon}(\textbf{y}^{k,\epsilon} ) =& \Delta J_k^2n ^{k,\epsilon} -  J_k(\nabla\cdot(J_kn^{k,\epsilon}\nabla (c^{k,\epsilon}*\rho^\epsilon)))-J_k(J_ku^{k,\epsilon}\cdot \nabla J_kn^{k,\epsilon}),\\
F_2^{k,\epsilon}(\textbf{y}^{k,\epsilon} )= &\Delta J_k^2c^{k,\epsilon}-J_k(J_kc^{k,\epsilon}(n^{k,\epsilon}*\rho^{\epsilon}))-J_k(J_ku^{k,\epsilon}\cdot \nabla J_kc^{k,\epsilon}),\\
F_3^{k,\epsilon}(\textbf{y}^{k,\epsilon} )= &-A J_k^2u^{k,\epsilon}+J_k\textbf{P}[ (n^{k,\epsilon}\nabla \phi)*\rho^\epsilon]-J_k\textbf{P} [(J_ku^{k,\epsilon}\cdot \nabla) J_ku^{k,\epsilon}].
\end{split}
\end{equation*}

\begin{lemma}\label{lem1}
Let $s \geq 100$ be an integer, $0 < \epsilon < 1$, and $k \in \mathbb{N}^+$. Under the same assumptions as in Theorem~\ref{th1}, the system \eqref{Mod-22} admits a unique solution
$$
(n^{k,\epsilon}, c^{k,\epsilon}, u^{k,\epsilon}) \in \mathcal{C}^1([0,T]; \textbf{H}^s(\mathbb{R}^2)).
$$
\end{lemma}

\begin{proof}[\emph{\textbf{Proof}}]
Actually, the functional $F^{k,\epsilon}(\cdot): \textbf{H}^s(\mathbb{R}^2) \to \textbf{H}^s(\mathbb{R}^2)$ is locally Lipschitz continuous. That is, for any given $M > 0$,
\begin{equation}\label{2.3}
\|F^{k,\epsilon}(\textbf{y}_1^{k,\epsilon}) - F^{k,\epsilon}(\textbf{y}_2^{k,\epsilon})\|_{\textbf{H}^s} \lesssim_{M,k,s,\phi,\epsilon} \|\textbf{y}_1^{k,\epsilon} - \textbf{y}_2^{k,\epsilon}\|_{\textbf{H}^s},
\end{equation}
where $\|\textbf{y}_1^{k,\epsilon}\|_{\textbf{H}^s} \leq M$ and $\|\textbf{y}_2^{k,\epsilon}\|_{\textbf{H}^s} \leq M$. As an illustration, we present the estimate for $F^{k,\epsilon}_3(\cdot)$, noting that the arguments for $F^{k,\epsilon}_1(\cdot)$ and $F^{k,\epsilon}_2(\cdot)$ are similar. Since $\operatorname{supp} \widehat{J_k f} \subset \{\xi \in \mathbb{R}^2 : k^{-1} \leq |\xi| \leq k\}$ for $k \geq 1$, it follows from the Bernstein inequality~\cite{bahouri2011fourier} and the Sobolev embedding $H^2(\mathbb{R}^2) \hookrightarrow L^{\infty}(\mathbb{R}^2)$ that
\begin{equation*}
\begin{split}
&\|F^{k,\epsilon}_3(\textbf{y}_1^{k,\epsilon}) - F^{k,\epsilon}_3(\textbf{y}_2^{k,\epsilon})\|_{\textbf{H}^s} \\
&\lesssim \|J_k^2 \Delta (u^{k,\epsilon}_1 - u^{k,\epsilon}_2)\|_{H^s} + \|J_k[((n^{k,\epsilon}_1 - n^{k,\epsilon}_2)\nabla \phi) * \rho^\epsilon]\|_{H^s} \\
&\quad + \|J_k[(J_k u^{k,\epsilon}_1 \cdot \nabla) J_k u^{k,\epsilon}_1 - (J_k u^{k,\epsilon}_2 \cdot \nabla) J_k u^{k,\epsilon}_2]\|_{H^s} \\
&\lesssim C_{k,s} \|u^{k,\epsilon}_1 - u^{k,\epsilon}_2\|_{H^s} + C_{k,s} \|\nabla \phi\|_{L^{\infty}} \|n^{k,\epsilon}_1 - n^{k,\epsilon}_2\|_{L^2} \\
&\quad + C_{k,s} \|(J_k u^{k,\epsilon}_1 \cdot \nabla) J_k u^{k,\epsilon}_1 - (J_k u^{k,\epsilon}_2 \cdot \nabla) J_k u^{k,\epsilon}_2\|_{L^2} \\
&\lesssim C_{k,s,\phi} \|\textbf{y}_1^{k,\epsilon} - \textbf{y}_2^{k,\epsilon}\|_{\textbf{H}^s} + C_{k,s} \|(J_k u^{k,\epsilon}_1 \cdot \nabla) J_k u^{k,\epsilon}_1 - (J_k u^{k,\epsilon}_1 \cdot \nabla) J_k u^{k,\epsilon}_2\|_{L^2} \\
&\quad + C_{k,s} \|(J_k u^{k,\epsilon}_1 \cdot \nabla) J_k u^{k,\epsilon}_2 - (J_k u^{k,\epsilon}_2 \cdot \nabla) J_k u^{k,\epsilon}_2\|_{L^2} \\
&\lesssim_{k,s,\phi} \|\textbf{y}_1^{k,\epsilon} - \textbf{y}_2^{k,\epsilon}\|_{\textbf{H}^s} + (\|J_k u^{k,\epsilon}_1\|_{L^{\infty}}^2 + \|J_k u^{k,\epsilon}_2\|_{L^{\infty}}^2) \|J_k u^{k,\epsilon}_1 - J_k u^{k,\epsilon}_2\|_{L^2} \\
&\lesssim_{M,k,s,\phi} \|\textbf{y}_1^{k,\epsilon} - \textbf{y}_2^{k,\epsilon}\|_{\textbf{H}^s}.
\end{split}
\end{equation*}
The Lipschitz estimate \eqref{2.3} permits the application of the Picard theorem in Banach spaces, yielding a unique local solution $\textbf{y}^{k,\epsilon} \in \mathcal{C}^1([0,T_k); \textbf{H}^s(\mathbb{R}^2))$, where $T_k$ denotes the maximal existence time. We now show that $T_k = \infty$. Following standard arguments, it can be shown that $n^{k,\epsilon}(t,x) \geq 0$ for all $(t,x) \in [0,T_k) \times \mathbb{R}^2$ (cf.~\cite{nie2020global}). By integration by parts, the second equation in \eqref{Mod-22} yields
\begin{equation}\label{qq1}
\frac{1}{2} \frac{\mathrm{d}}{\mathrm{d}t} \|c^{k,\epsilon}\|_{L^2}^2 + \|\nabla J_k c^{k,\epsilon}\|_{L^2}^2 = -\int_{\mathbb{R}^2} |J_k c^{k,\epsilon}(x)|^2 (n^{k,\epsilon} * \rho^\epsilon)(x) \,\mathrm{d}x \leq 0.
\end{equation}
Similarly,
\begin{equation}\label{qq2}
\begin{split}
\frac{1}{2} \frac{\mathrm{d}}{\mathrm{d}t} \|n^{k,\epsilon}\|_{L^2}^2 + \|\nabla J_k n^{k,\epsilon}\|_{L^2}^2 &\leq \frac{1}{2} \|\nabla J_k n^{k,\epsilon}\|_{L^2}^2 + C_\epsilon \|c^{k,\epsilon}\|_{L^2}^2 \|n^{k,\epsilon}\|_{L^2}^2.
\end{split}
\end{equation}
And
\begin{equation}\label{qq3}
\frac{1}{2} \frac{\mathrm{d}}{\mathrm{d}t} \|u^{k,\epsilon}\|_{L^2}^2 + \|\nabla J_k u^{k,\epsilon}\|_{L^2}^2 \lesssim \|\nabla \phi\|_{L^{\infty}}^2 \|n^{k,\epsilon}\|_{L^2}^2 + \|u^{k,\epsilon}\|_{L^2}^2.
\end{equation}
Applying Gronwall's lemma to \eqref{qq1}--\eqref{qq3} gives
\begin{equation}\label{lem2-4}
\|\textbf{y}^{k,\epsilon}(t)\|_{\textbf{H}^0}^2 + \int_0^t \|J_k \textbf{y}^{k,\epsilon}(r)\|_{\textbf{H}^1}^2 \,\mathrm{d}r \lesssim_{t,\epsilon,\phi,\|\textbf{y}^{\epsilon}(0)\|_{\textbf{H}^0}} 1.
\end{equation}
Moreover, using the Bernstein inequality, we deduce
\begin{equation*}
\frac{\mathrm{d}}{\mathrm{d}t} \|\textbf{y}^{k,\epsilon}(t)\|_{\textbf{H}^s}^2 \lesssim \|F^{k,\epsilon}(\textbf{y}^{k,\epsilon}(t))\|_{\textbf{H}^s}^2 + \|\textbf{y}^{k,\epsilon}(t)\|_{\textbf{H}^s}^2 \lesssim_{t,\epsilon,\phi,k,s} C_{\|\textbf{y}^{k,\epsilon}(t)\|_{\textbf{H}^0}} \|\textbf{y}^{k,\epsilon}(t)\|_{\textbf{H}^s}^2.
\end{equation*}
Applying Gronwall's lemma again and using \eqref{lem2-4}, we obtain
\begin{equation*}
\|\textbf{y}^{k,\epsilon}(t)\|_{\textbf{H}^s}^2 \lesssim_{t,\epsilon,\phi,k,s,\|\textbf{y}^{\epsilon}(0)\|_{\textbf{H}^0}} 1, \quad \forall t \in [0,T_k).
\end{equation*}
Hence, invoking the continuation criterion for ODEs in Banach spaces (cf.~\cite[Theorem~3.3]{majda2002vorticity}), we conclude that the solution exists globally in time. This completes the proof.
\end{proof}

Next, we shall establish the uniform $\textbf{H}^s$-estimates for $\textbf{y}^{k,\epsilon}$ (independent of k) and successively take the limit as $k\rightarrow\infty$ to prove that the limit process $\textbf{y}^{\epsilon}:=(n^{\epsilon},c^{\epsilon},u^{\epsilon})$ is the unique global sufficiently smooth solution to the system \eqref{Mod-1}.
\begin{lemma}\label{lem2} Let integer $s\geq100$ and $1>\epsilon>0$. Under the same assumptions as in Theorem \ref{th1}, the system \eqref{Mod-1} admits a unique global solution
\begin{equation}\label{re}
\begin{split}
\mathbf{y}^{\epsilon}\in \mathcal{C}([0,T];\textbf{H}^{s-3}(\mathbb{R}^2))\cap \mathcal{C}^1([0,T];\textbf{H}^{s-5}(\mathbb{R}^2)).
\end{split}
\end{equation}
\end{lemma}

\begin{proof}[\emph{\textbf{Proof}}]
\textsf{Step 1: $\mathbf{H}^s$-estimate.} The $\textbf{H}^0$-estimate for the solution $\textbf{y}^{k,\epsilon}$ has already been established in \eqref{lem2-4}. We now proceed to derive the $\textbf{H}^1$-estimate. Applying the Gagliardo-Nirenberg inequality
\begin{equation*}
\begin{split}
\|f\|_{L^4(\mathbb{R}^2)}\lesssim\|\nabla f\|_{L^2(\mathbb{R}^2)}^{\frac{1}{2}}\|f\|_{L^2(\mathbb{R}^2)}^{\frac{1}{2}},\quad f\in H^1(\mathbb{R}^2),
\end{split}
\end{equation*}
we obtain
\begin{equation}\label{lem2-5}
\begin{split}
&\frac{1}{2}\frac{\mathrm{d}}{\mathrm{d}t}\|\nabla u^{k,\epsilon}\|_{L^2}^2+\|\Delta J_ku^{k,\epsilon}\|_{L^2}^2\\
&=((J_ku^{k,\epsilon}\cdot \nabla) J_ku^{k,\epsilon},\Delta J_ku^{k,\epsilon})_{L^2}-((n^{k,\epsilon}\nabla \phi)*\rho^\epsilon,\Delta J_ku^{k,\epsilon})_{L^2}\\
&\leq\frac{1}{2}\|\Delta J_ku^{k,\epsilon}\|_{L^2}^2+\|J_ku^{k,\epsilon}\|_{L^4}^2\|\nabla J_ku^{k,\epsilon}\|_{L^4}^2+\|\nabla \phi\|_{L^{\infty}}^2\|J_kn^{k,\epsilon}\|_{L^2}^2\\
&\leq\frac{3}{4}\|\Delta J_ku^{k,\epsilon}\|_{L^2}^2+C\|J_ku^{k,\epsilon}\|_{L^2}^2\|\nabla J_ku^{k,\epsilon}\|_{L^2}^2\|\nabla u^{k,\epsilon}\|_{L^2}^2+\|\nabla \phi\|_{L^{\infty}}^2\|J_kn^{k,\epsilon}\|_{L^2}^2.
\end{split}
\end{equation}
Invoking the Gronwall lemma, and using the estimates \eqref{lem2-4} and \eqref{lem2-5}, we deduce that
\begin{equation}\label{lem2-7}
\begin{split}
\|\nabla u^{k,\epsilon}(t)\|_{L^2}^2+\int_0^t\|\Delta J_ku^{k,\epsilon}(r)\|_{L^2}^2\mathrm{d}r\lesssim_{t,\epsilon,\phi,\|\textbf{y}^{\epsilon}(0)\|_{\textbf{H}^1}}1.
\end{split}
\end{equation}
Similarly, using the Sobolev embedding $H^2(\mathbb{R}^2)\hookrightarrow L^{\infty}(\mathbb{R}^2)$, we infer that
\begin{equation}\label{lem2-8}
\begin{split}
&\frac{1}{2}\frac{\mathrm{d}}{\mathrm{d}t}\|\nabla n^{k,\epsilon}\|_{L^2}^2+\|\Delta J_kn^{k,\epsilon}\|_{L^2}^2\\
&=(J_ku^{k,\epsilon}\cdot \nabla J_kn^{k,\epsilon},\Delta J_kn^{k,\epsilon})_{L^2}+(\nabla\cdot(J_kn^{k,\epsilon}\nabla (c^{k,\epsilon}*\rho^\epsilon)),\Delta J_kn^{k,\epsilon})_{L^2}\\
&\leq\frac{1}{2}\|\Delta J_kn^{k,\epsilon}\|_{L^2}^2+C\|J_ku^{k,\epsilon}\|_{L^{\infty}}^2\|\nabla J_kn^{k,\epsilon}\|_{L^2}^2+C\|\nabla(c^{k,\epsilon}*\rho^\epsilon)\|_{L^{\infty}}^2\|\nabla J_kn^{k,\epsilon}\|_{L^2}^2\\
&\quad+C\|\Delta(c^{k,\epsilon}*\rho^\epsilon)\|_{L^{\infty}}^2\|J_kn^{k,\epsilon}\|_{L^2}^2\\
&\leq\frac{1}{2}\|\Delta J_kn^{k,\epsilon}\|_{L^2}^2+C\|J_ku^{k,\epsilon}\|_{H^{2}}^2\|\nabla n^{k,\epsilon}\|_{L^2}^2+C_{\epsilon}\|c^{k,\epsilon}\|_{L^{2}}^2\|\nabla n^{k,\epsilon}\|_{L^2}^2+C_{\epsilon}\|c^{k,\epsilon}\|_{L^{2}}^2\|n^{k,\epsilon}\|_{L^2}^2.
\end{split}
\end{equation}
Applying the Gronwall lemma once again, we derive from \eqref{lem2-4}, \eqref{lem2-7}, and \eqref{lem2-8} that
\begin{equation*}
\begin{split}
\|\nabla n^{k,\epsilon}(t)\|_{L^2}^2+\int_0^t\|\Delta J_kn^{k,\epsilon}(r)\|_{L^2}^2\mathrm{d}r\lesssim_{t,\epsilon,\phi,\|\textbf{y}^{\epsilon}(0)\|_{\textbf{H}^1}}1.
\end{split}
\end{equation*}
By a similar argument applied to the $c^{k,\epsilon}$-equation, we deduce
\begin{equation*}
\begin{split}
\|\nabla c^{k,\epsilon}(t)\|_{L^2}^2+\int_0^t\|\Delta J_kc^{k,\epsilon}(r)\|_{L^2}^2\mathrm{d}r\lesssim_{t,\epsilon,\phi,\|\textbf{y}^{\epsilon}(0)\|_{\textbf{H}^1}}1.
\end{split}
\end{equation*}
Consequently, we obtain the uniform $\textbf{H}^1$-bound:
\begin{equation}\label{lem2-9}
\begin{split}
\|\textbf{y}^{k,\epsilon}(t)\|_{\textbf{H}^1}^2+\int_0^t\|J_k\textbf{y}^{k,\epsilon}(r)\|_{\textbf{H}^2}^2\mathrm{d}r\lesssim_{t,\epsilon,\phi,\|\textbf{y}^{\epsilon}(0)\|_{\textbf{H}^1}}1.
\end{split}
\end{equation}

We now establish the higher-order estimate for $\textbf{y}^{k,\epsilon}$ in $\textbf{H}^s$ with $s\geq 100$. Applying the Bessel potential operator $\Lambda^s$ to the $n^{k,\epsilon}$-equation in \eqref{mod-22} and taking the $L^2$-inner product with $\Lambda^s n^{k,\epsilon}$ yields
\begin{equation}\label{lem2-101}
\begin{split}
&\frac{1}{2}\frac{\mathrm{d}}{\mathrm{d}t}\|n^{k,\epsilon}\|_{H^s}^2+\|\nabla J_kn^{k,\epsilon}\|_{H^s}^2\\
&=(\Lambda^s(J_ku^{k,\epsilon} J_kn^{k,\epsilon}),\nabla\Lambda^s J_kn^{k,\epsilon})_{L^2}+(\Lambda^s(J_kn^{k,\epsilon}\nabla (c^{k,\epsilon}*\rho^\epsilon)),\nabla\Lambda^s J_kn^{k,\epsilon})_{L^2}\\
&\leq\frac{1}{2}\|\nabla\Lambda^s J_kn^{k,\epsilon}\|_{L^2}^2+\|\Lambda^s(J_ku^{k,\epsilon} J_kn^{k,\epsilon})\|_{L^2}^2+\|\Lambda^s(J_kn^{k,\epsilon}\nabla (c^{k,\epsilon}*\rho^\epsilon))\|_{L^2}^2.
\end{split}
\end{equation}
Noting that $H^m(\mathbb{R}^2)$ is a Banach algebra for $m > 1$, we deduce from \eqref{lem2-101} that
\begin{equation}\label{lem2-10}
\begin{split}
\frac{\mathrm{d}}{\mathrm{d}t}\|n^{k,\epsilon}\|_{H^s}^2 + \|\nabla J_k n^{k,\epsilon}\|_{H^s}^2 \lesssim_s \left( \|J_k u^{k,\epsilon}\|_{H^s}^2 + C_\epsilon \|c^{k,\epsilon}\|_{L^2}^2 \right) \|J_k n^{k,\epsilon}\|_{H^s}^2.
\end{split}
\end{equation}
Similarly, we obtain
\begin{equation}\label{cc}
\begin{split}
\frac{\mathrm{d}}{\mathrm{d}t}\|c^{k,\epsilon}\|_{H^s}^2 + \|\nabla J_k c^{k,\epsilon}\|_{H^s}^2 \lesssim_s \left( \|J_k u^{k,\epsilon}\|_{H^s}^2 + C_\epsilon \|J_k n^{k,\epsilon}\|_{L^2}^2 \right) \|J_k c^{k,\epsilon}\|_{H^s}^2.
\end{split}
\end{equation}
By employing the Moser-type estimate (cf. \cite{miao2012littlewood}),
\begin{equation*}
\begin{split}
\|fg\|_{H^s} \lesssim_s \|f\|_{L^\infty} \|g\|_{H^s} + \|g\|_{L^\infty} \|f\|_{H^s}, \quad s > 0, ~ f,~g \in L^\infty(\mathbb{R}^2) \cap H^s(\mathbb{R}^2),
\end{split}
\end{equation*}
we infer from the $u^{k,\epsilon}$-equation that
\begin{equation}\label{u2}
\begin{split}
\frac{\mathrm{d}}{\mathrm{d}t} \|u^{k,\epsilon}\|_{H^s}^2 + \|\nabla J_k u^{k,\epsilon}\|_{H^s}^2
&\lesssim \|J_k u^{k,\epsilon} \otimes J_k u^{k,\epsilon}\|_{H^s}^2 + \|(n^{k,\epsilon} \nabla \phi) * \rho^\epsilon\|_{H^s}^2 \\
&\lesssim_s \|J_k u^{k,\epsilon}\|_{H^2}^2 \|J_k u^{k,\epsilon}\|_{H^s}^2 + C_\epsilon \|\nabla \phi\|_{L^\infty}^2 \|n^{k,\epsilon}\|_{L^2}^2.
\end{split}
\end{equation}
Next, multiplying both sides of the third equation in \eqref{mod-22} by $-\Delta^2 u^{k,\epsilon}$ yields
\begin{equation}\label{u11}
\begin{split}
\frac{\mathrm{d}}{\mathrm{d}t} \|\Delta u^{k,\epsilon}\|_{L^2}^2 + \|\nabla J_k \Delta u^{k,\epsilon}\|_{L^2}^2
&\lesssim \|\nabla((J_k u^{k,\epsilon} \cdot \nabla) J_k u^{k,\epsilon})\|_{L^2}^2 + \|\nabla((n^{k,\epsilon} \nabla \phi) * \rho^\epsilon)\|_{L^2}^2 \\
&\lesssim \|J_k u^{k,\epsilon}\|_{H^2}^2 \|\Delta u^{k,\epsilon}\|_{L^2}^2 + C_\epsilon \|\nabla \phi\|_{L^\infty}^2 \|n^{k,\epsilon}\|_{L^2}^2.
\end{split}
\end{equation}
By Gronwall's lemma, it follows from \eqref{lem2-9} and \eqref{u11} that
\begin{equation}\label{u111}
\begin{split}
\|u^{k,\epsilon}(t)\|_{H^2}^2 + \int_0^t \|J_k u^{k,\epsilon}(r)\|_{H^3}^2 \,\mathrm{d}r \lesssim_{t,\epsilon,\phi,\|\textbf{y}^{\epsilon}(0)\|_{\textbf{H}^2}} 1.
\end{split}
\end{equation}
Applying Gronwall's lemma once again, we deduce from \eqref{lem2-9}, \eqref{u2}, and \eqref{u111} that
\begin{equation}\label{use}
\begin{split}
\|u^{k,\epsilon}(t)\|_{H^s}^2 + \int_0^t \|J_k u^{k,\epsilon}(r)\|_{H^{s+1}}^2 \,\mathrm{d}r \lesssim_{t,\epsilon,\phi,s,\|\textbf{y}^{\epsilon}(0)\|_{\textbf{H}^s}} 1.
\end{split}
\end{equation}
Finally, combining the estimate \eqref{use} with Gronwall's lemma applied to \eqref{lem2-10}, and \eqref{cc}, we arrive at
\begin{equation}\label{lem2-14}
\begin{split}
\|\textbf{y}^{k,\epsilon}(t)\|_{\textbf{H}^s}^2 + \int_0^t \|J_k \textbf{y}^{k,\epsilon}(r)\|_{\textbf{H}^{s+1}}^2 \,\mathrm{d}r \lesssim_{t,\epsilon,\phi,s,\|\textbf{y}^{\epsilon}(0)\|_{\textbf{H}^s}} 1.
\end{split}
\end{equation}

\textsf{Step 2: Strong convergence result.} We shall prove that
\begin{equation}\label{qqq1}
\begin{split}
\textbf{y}^{k,\epsilon}\rightarrow \textbf{y}^{\epsilon}~\textrm{strongly in}~\mathcal {C}([0,T];\textbf{H}^{s-3}(\mathbb{R}^2))~\textrm{as}~k\rightarrow\infty.
\end{split}
\end{equation}
To this end, for any fixed $\epsilon \in (0,1)$, let $\textbf{y}^{k,\epsilon}$ and $\textbf{y}^{l,\epsilon}$ be solutions to \eqref{Mod-22} associated to the frequency truncations $J_k$ and $J_l$, respectively. Define $\textbf{y}^{k,l,\epsilon}:=\textbf{y}^{k,\epsilon}-\textbf{y}^{l,\epsilon}$, it follows from \eqref{Mod-22} that
\begin{equation}\label{2..1}
\left\{
\begin{aligned}
& \frac{\mathrm{d}}{\mathrm{d} t}  \textbf{y}^{k,l,\epsilon}= F ^{k,\epsilon} (\textbf{y}^{k,\epsilon}  )-F ^{l,\epsilon} (\textbf{y}^{l,\epsilon}  ) ,\\
& \textbf{y}^{k,l,\epsilon}(0) =0.
\end{aligned}
\right.
\end{equation}
For convenience,  in the following discussion we will simply write $\textbf{y}^{k,l}$, $\textbf{y}^{k}$, $\textbf{y}$ and $F ^{k}(\cdot)$ instead of $\textbf{y}^{k,l,\epsilon}$, $\textbf{y}^{k,\epsilon}$, $\textbf{y}^{\epsilon}$ and $F ^{k,\epsilon}(\cdot)$, respectively. Then, we have
\begin{equation*}
\begin{split}
F ^{k}_1 (\textbf{y}^{k}  )-F ^{l}_1 (\textbf{y}^{l}  )&=\Delta J_k^2  n^{k,l}+(J_k^2-J_l^2)\Delta n^l+[J_l(J_l u^l\cdot \nabla J_l n^l)- J_k(J_k u^k\cdot \nabla J_k n^k)]   \\
&\quad+[\nabla\cdot J_l(J_ln^l(\nabla c^l*\rho^{\epsilon}))- \nabla\cdot J_k(J_kn^k(\nabla c^k*\rho^{\epsilon}))]\\
&:= \Delta J_k^2  n^{k,l}+ I_1+I_2+I_3,
 \end{split}
\end{equation*}
\begin{equation*}
\begin{split}
F ^{k}_2 (\textbf{y}^{k}  )-F ^{l}_2 (\textbf{y}^{l}  )&=\Delta J_k^2  c^{k,l}+(J_k^2-J_l^2)\Delta c^l+[J_l(J_l u^l\cdot \nabla J_l c^l)- J_k(J_k u^k\cdot \nabla J_k c^k)]\\
&\quad+[J_l(J_lc^l(J_ln^l*\rho^{\epsilon}))-J_k(J_kc^k(J_kn^k*\rho^{\epsilon}))]\\
&:=\Delta J_k^2  c^{k,l}+ K_1+K_2 + K_3,
 \end{split}
\end{equation*}
and
\begin{equation*}
\begin{split}
F ^{k}_3 (\textbf{y}^{k}  )-F ^{l}_3 (\textbf{y}^{l}  )&=\Delta J_k^2  u^{k,l}+(J_k^2-J_l^2)\Delta u^l + \textbf{P} [  J_l((J_l u^l\cdot \nabla) J_l u^l)- J_k((J_k u^k\cdot \nabla) J_k u^k)]      \\
&\quad+ \textbf{P} [J_k((J_kn^k\nabla \phi)*\rho^{\epsilon})- J_l((J_ln^l\nabla \phi)*\rho^{\epsilon})]\\
&:=\Delta J_k^2  u^{k,l}+ L_1+L_2 + L_3.
 \end{split}
\end{equation*}
Actually, for the terms associated to $I_i,~K_i$ and $L_i$, we have
\begin{subequations}
\begin{align}
& |(n^{k,l},I_1)_{H^{s-3}}| \lesssim \|n^{k,l}\|_{H^{s-3}}^2+\max\{\frac{1}{k^2},\frac{1}{l^2}\} \|  n^l\|_{H^{s}}^2 ,\label{2.30a}\\
&|(n^{k,l},I_2)_{H^{s-3}}|\lesssim\|n^{k,l}\|_{H^{s-3}} ^2+\max\{\frac{1}{k^2},\frac{1}{l^2}\} (\| u^k\|_{H^{s}}^4+\| u^l\|_{H^{s}}^4+\| n^l\|_{H^{s}}^4+\| n^k\|_{H^{s}} ^4),\label{2.30b}\\
&|(n^{k,l},I_3)_{H^{s-3}}| \lesssim\|n^{k,l}\|_{H^{s-3}} ^2   +  \max\{\frac{1}{k^2},\frac{1}{l^2}\} (\|n^k \|_{H^{s}}^4+\|n^l \|_{H^{s}}^4+\|c^l \|_{H^{s}}^4+\|c^k\|_{H^{s}}^4),\label{2.30c}\\
& |(c^{k,l},K_1)_{H^{s-3}}| \lesssim \|c^{k,l}\|_{H^{s-3}}^2+\max\{\frac{1}{k^2},\frac{1}{l^2}\} \|c^l\|_{H^{s}}^2 \label{2.31a},\\
&|(c^{k,l},K_2)_{H^{s-3}}| \lesssim \|c^{k,l}\|_{H^{s-3}}^2 + \frac{1}{k^2}(\| u^k\|_{H^{s}}^4+\| u^l\|_{H^{s}}^4+\| c^l\|_{H^{s}}^4+\| c^k\|_{H^{s}} ^4),\label{2.31b}\\
&|(c^{k,l},K_3)_{H^{s-3}}| \lesssim \| c^{k,l} \|_{H^{s-3}}^2 + \max\{\frac{1}{k^2},\frac{1}{l^2}\}(\|n^k \|_{H^{s}}^4+\|n^l \|_{H^{s}}^4+\|c^l \|_{H^{s}}^4+\|c^k\|_{H^{s}}^4),\label{2.31c}\\
& |(u^{k,l},L_1)_{H^{s-3}}| \lesssim \|u^{k,l}\|_{H^{s-3}}^2+\max\{\frac{1}{k^2},\frac{1}{l^2}\} \|  u^l\|_{H^{s}}^2 ,\label{2.32a}\\
&|(u^{k,l},L_2)_{H^{s-3}}| \lesssim  \|u^{k,l}\|_{H^{s-3}}^2+ \max\{\frac{1}{k^2},\frac{1}{l^2}\}(\|u^{k}\|_{H^{s}}^4+\|u^{l}\|_{H^{s}}^4 ),\label{2.32b}\\
&|(u^{k,l},L_3)_{H^{s-3}}|\lesssim_{\phi,\epsilon}  \|u^{k,l}\|_{H^{s-3}}^2+ \max\{\frac{1}{k^2},\frac{1}{l^2}\} (\|n^k \|_{H^s}^2+\|n^l \|_{H^s}^2).\label{2.32c}
\end{align}
\end{subequations}
Generally, it follows from the Bernstein inequality that
\begin{equation*}
\begin{split}
\|J_l f^l- J_k f^k\|_{H^{s-d}}\lesssim\left\{
\begin{aligned}
& \max\{\frac{1}{k^d},\frac{1}{l^d}\} ( \|f^k\|_{H^{s}}+\|f^l\|_{H^{s}}),\\
&   \max\{\frac{1}{k^d},\frac{1}{l^d}\} \|f^l\|_{H^{s}}+ \|f^{k,l}\|_{H^{s-d}} ,
\end{aligned}
\right.\quad  \textrm{for all} ~s \in \mathbb{R},~d\in \mathbb{N}.
 \end{split}
\end{equation*}
For \eqref{2.30a}, by using the Bernstein inequality, we get
\begin{equation*}
\begin{split}
 |(n^{k,l},I_1)_{H^{s-3}}|&\leq \|n^{k,l}\|_{H^{s-3}}\|(J_k +J_l )(J_k -J_l )\Delta n^l\|_{H^{s-3}}\\
 & \lesssim \max\{\frac{1}{k},\frac{1}{l}\} \|n^{k,l}\|_{H^{s-3}}\|  n^l\|_{H^{s}} \lesssim \|n^{k,l}\|_{H^{s-3}}^2+\max\{\frac{1}{k^2},\frac{1}{l^2}\} \|  n^l\|_{H^{s}}^2.
 \end{split}
\end{equation*}
For \eqref{2.30b}, by the fact that $H^{s-3}(\mathbb{R}^2)$ is a Banach algebra, we have
\begin{equation*}
\begin{split}
 |(n^{k,l},I_2)_{H^{s-3}}|&\leq\Bigl[|(n^{k,l},(J_l-J_k)(J_l u^l\cdot \nabla J_l n^l))_{H^{s-3}}|\\
 &\quad+|(n^{k,l},J_k[ (J_l u^l-J_ku^k)\cdot \nabla J_l n^l +  J_k u^k\cdot \nabla( J_l n^l- J_k n^k) ] )_{H^{s-3}}|\Bigl]\\
 & \lesssim  \|n^{k,l}\|_{H^{s-3}} \left(\frac{1}{l^5}\| u^l\|_{H^{s}}\| n^l\|_{H^{s}}+ \max\{\frac{1}{k^3l^2},\frac{1}{l^5} \}(\|u^k\|_{H^{s}}+\|u^l\|_{H^{s}})\| n^l \|_{H^{s}}\right.\\
 &\left.\quad+ \max\{\frac{1}{k^3l^2},\frac{1}{k^5}\} \|  u^k\|_{H^{s}} (\|n^k\|_{H^{s}}+\|n^l\|_{H^{s}})\right)\\
 &\lesssim  \|n^{k,l}\|_{H^{s-3}} ^2+  \max\{\frac{1}{l^{2}}, \frac{1}{k^{2}} \}\left( \| u^k\|_{H^{s}}^4+\| u^l\|_{H^{s}}^4+\| n^l\|_{H^{s}}^4+\| n^k\|_{H^{s}} ^4\right).
 \end{split}
\end{equation*}
Similarly, it follows that
\begin{equation*}
\begin{split}
|(n^{k,l},I_3)_{H^{s-3}}|&\lesssim \|n^{k,l}\|_{H^{s-3}} \left(\|(J_l-J_k) J_ln^l\nabla (c^l*\rho^{\epsilon})\|_{H^{s-2}}+\|(J_l n^l-J_k n^k) \nabla  (c^k*\rho^{\epsilon})\|_{H^{s-2}}\right.\\
 &\left.\quad+\| J_k n^k [\nabla (c^l*\rho^{\epsilon})-\nabla (c^k*\rho^{\epsilon})]\|_{H^{s-2}}\right)\\
 &\lesssim\|n^{k,l}\|_{H^{s-3}} \left( \frac{1}{l^2} \| n^l \|_{H^{s}}\|c^l\|_{H^{s}}  +  \max\{\frac{1}{k^2},\frac{1}{l^2}\} ( \|n^k\|_{H^{s}}+\|n^l\|_{H^{s}})\|c^k\|_{H^{s}}\right.\\
 &\left.\quad+   \max\{\frac{1}{k^2},\frac{1}{l^2}\} \| n^k\|_{H^{s}}(\|c^l \|_{H^{s}}+\|c^k\|_{H^{s}})\right)\\
 &\lesssim\|n^{k,l}\|_{H^{s-3}} ^2   +  \max\{\frac{1}{k^2},\frac{1}{l^2}\} (\|n^k \|_{H^{s}}^4+\|n^l \|_{H^{s}}^4+\|c^l \|_{H^{s}}^4+\|c^k\|_{H^{s}}^4).
  \end{split}
\end{equation*}
The proofs of inequalities \eqref{2.31a} through \eqref{2.32c} are similar and are therefore omitted. We now proceed to estimate $\|\Lambda^{s-3}\textbf{y}^{k,l}\|_{L^2}^2$ based on the system \eqref{2..1}. Utilizing the estimates \eqref{2.30a}--\eqref{2.32c}, we infer that
\begin{equation}\label{2.31}
\begin{split}
& \|\textbf{y}^{k,l}(t)\|_{\textbf{H}^{s-3}}^2 + 2 \int_0^t\|\nabla J_k \textbf{y}^{k,l}(r)\|_{\textbf{H}^{s-3}}^2 \mathrm{d}r \\
&= 2 \sum_{1\leq i \leq 3}\int_0^t (n^{k,l}(r), I_i(r))_{H^{s-3}} \mathrm{d}r + 2 \sum_{1\leq i \leq 3}\int_0^t (c^{k,l}(r), K_i(r))_{H^{s-3}} \mathrm{d}r \\
&\quad + 2 \sum_{1\leq i \leq 3}\int_0^t (u^{k,l}(r), L_i(r))_{H^{s-3}} \mathrm{d}r \\
&\lesssim_{\phi,\epsilon,s} \int_0^t \|\textbf{y}^{k,l}(r)\|_{\textbf{H}^{s-3}}^2 \mathrm{d}r + \max\left\{\frac{1}{k^2}, \frac{1}{l^2}\right\} \int_0^t \left( 1 + \|\textbf{y}^k(r)\|_{\textbf{H}^s}^4 + \|\textbf{y}^l(r)\|_{\textbf{H}^s}^4 \right) \mathrm{d}r.
\end{split}
\end{equation}
Applying Gronwall's lemma to \eqref{2.31}, and using the uniform estimate \eqref{lem2-14}, we obtain
\begin{equation*}
\begin{split}
\|\textbf{y}^{k,l}(t)\|_{\textbf{H}^{s-3}}^2 \lesssim_{t,\epsilon,\phi,s,\|\textbf{y}^{\epsilon}(0)\|_{\textbf{H}^s}} \max\left\{\frac{1}{k^2}, \frac{1}{l^2}\right\} \to 0 \quad \text{as} \quad k,~l \to \infty.
\end{split}
\end{equation*}
Therefore, the sequence $\{\textbf{y}^k\}_{k \in \mathbb{N}^+}$ is a Cauchy sequence in $\mathcal{C}([0,T];\textbf{H}^{s-3}(\mathbb{R}^2))$, which yields the convergence result \eqref{qqq1}. Since $\textbf{y} \in \mathcal{C}([0,T];\textbf{H}^{s-3}(\mathbb{R}^2))$ and the space $\textbf{H}^{s-5}(\mathbb{R}^2)$ is a Banach algebra, the parabolic structure of the system \eqref{Mod-1} implies that $\frac{\mathrm{d}}{\mathrm{d}t}\textbf{y} \in \mathcal{C}([0,T];\textbf{H}^{s-5}(\mathbb{R}^2))$, establishing the time-space regularity \eqref{re}.
Finally, the proof of uniqueness follows the same argument as in Proposition~\ref{Pro}, and we omit the details here. The proof of Lemma~\ref{lem2} is thus complete.

\end{proof}

\section{Proof of Theorem \ref{th1}}\label{sec3}
In this section, we first derive a uniform entropy-energy estimate for the family of smooth approximate solutions $\{\textbf{y}^{\epsilon}\}_{\epsilon > 0}$, where the estimate is independent of the regularization parameter $\epsilon$. Based on this uniform bound, we then establish the existence of a unique strong solution to the original system \eqref{CNS0} by passing to the limit $\epsilon \to 0$ using the weak compactness method. We begin with the following uniform a priori estimate.

\subsection{A entropy-energy estimate}\label{sec3-1}
\begin{lemma} \label{5lem}
Let the assumptions in Theorem \ref{th1} be satisfied. Let $(n^\epsilon,c^\epsilon,u^\epsilon)$ be the global solution of \eqref{Mod-1} with $\epsilon \in (0,1)$. Then the following estimates hold:
\begin{align}
&\|n^{\epsilon}(t)\|_{L^1}=\|n_0\|_{L^1};\quad\|c^{\epsilon}(t)\|_{L^1\cap L^{\infty}}\leq\|c_0\|_{L^1\cap L^{\infty}};\label{5lem-1}\\
&\mathcal{F}_1(n^{\epsilon},c^{\epsilon},u^{\epsilon})(t)
+ \int_0^{t}\mathcal{G}_1(n^{\epsilon},c^{\epsilon},u^{\epsilon})(r)\mathrm{d}r\lesssim _{n_0,c_0,u_0,\phi,t}  1.\label{5lem-2}
\end{align}
where
\begin{equation*}
\begin{split}
&\mathcal{F}_1(n^{\epsilon},c^{\epsilon},u^{\epsilon})(t):=\|n^{\epsilon}(t)\|_{L^1\cap L\ln L}+\|\nabla\sqrt{c^{\epsilon}(t)}\|_{L^2}^2+\|u^{\epsilon}(t)\|_{L^2}^2,\\
&\mathcal{G}_1(n^{\epsilon},c^{\epsilon},u^{\epsilon})(t):=\|\nabla\sqrt{n^{\epsilon}(t)+1}\|_{L^2}^2+\|\sqrt{c^{\epsilon}(t)}|D^2 \ln c^{\epsilon}(t)|\|_{L^2}^2+\|n^{\epsilon}(t)*\rho^{\epsilon}|\nabla\sqrt{c^{\epsilon}(t)}|^2\|_{L^1}+\|\nabla u^{\epsilon}(t)\|_{L^2}^2.
\end{split}
\end{equation*}
Here, $D^2 \ln c^\epsilon$ denotes the Hessian matrix of $\ln c^\epsilon$ and $L \ln L(\mathbb{R}^2)$ denotes the  Zygmund space which is equipped with the norm
$
\|f\|_{L \ln L}=\inf  \{k>0; \int_{\mathbb{R}^2} Z(f/ k) \mathrm{d}x \leq 1 \}
$
with respect to $Z(t) =t\ln^+ t$ if $t\geq1$ and $Z(t) =0$ otherwise. Moreover, it follows that
\begin{align}
&\|c^{\epsilon}(t)\|_{H^1}^2+\int_0^t\|c^{\epsilon}(r)\|_{H^2}^2\mathrm{d}r\lesssim _{n_0,c_0,u_0,\phi,t}  1,\label{5lem-3}\\
&\|n^{\epsilon}(t)\|_{L^2}^2+\int_0^t\|n^{\epsilon}(r)\|_{H^1}^2\mathrm{d}r\lesssim _{n_0,c_0,u_0,\phi,t}  1.\label{5lem-4}
\end{align}
\end{lemma}

\begin{proof}[\emph{\textbf{Proof}}]
Since $(n^{\epsilon},c^{\epsilon},u^{\epsilon})\in \mathcal {C}^1([0,T];H^{s-5}(\mathbb{R}^2))$ (cf. Lemma \ref{lem2}), it follows from the Sobolev embedding $H^{s-5}(\mathbb{R}^2)\hookrightarrow \mathcal{C}^{s-7}_b(\mathbb{R}^2)$ that $(n^{\epsilon}, c^{\epsilon},u^{\epsilon})$ satisfies the equations of \eqref{Mod-1} in the classical sense. By using the $L^2$-energy estimates for $n^{\epsilon}_-=\max\{-n^{\epsilon},0\}$ and $c^{\epsilon}_-=\max\{-c^{\epsilon},0\}$, it is standard to show that $n^{\epsilon}(t,x)\geq0$ and $c^{\epsilon}(t,x)\geq0$ (cf. \cite{nie2020global}). Moreover, due to the parabolic structure of the $n^{\epsilon},~c^{\epsilon}$-equations, it follows from the strong maximum principle (cf. \cite[Theorem 2.7]{lieberman1996second}) that $n^{\epsilon}(t,x)>0,~c^{\epsilon}(t,x)>0$ in $[0,T]\times \mathbb{R}^2$. Now let us derive the a priori uniform bound \eqref{5lem-2}. At first, based on the structure of $n^{\epsilon}$-equation and $c^{\epsilon}$-equation, it is straightforward to get \eqref{5lem-1}. By applying the chain rule to the function $(n^\epsilon+1)\ln(n^\epsilon+1)$ and integrating by parts, we obtain
\begin{equation} \begin{split}\label{lem3-1}
&\frac{\mathrm{d}}{\mathrm{d}t}\|(n^\epsilon+1)\ln(n^\epsilon+1)\|_{L^1} + 4\|\nabla\sqrt{n^\epsilon+1}\|^2_{L^2} \\
&= \int_{\mathbb{R}^2}\nabla n^\epsilon(x)\cdot\nabla (c^{\epsilon}*\rho^{\epsilon}(x))\,\mathrm{d}x + \int_{\mathbb{R}^2}\Delta ( c^{\epsilon}*\rho^{\epsilon}(x))\ln(n^\epsilon(x)+1)\,\mathrm{d}x.
\end{split} \end{equation}
Note that $\Delta c^{\epsilon} = 2|\nabla\sqrt{c^{\epsilon}}|^2 + 2\sqrt{c^{\epsilon}}\Delta\sqrt{c^{\epsilon}}$, so the second equation in \eqref{Mod-1} can be rewritten as
\begin{equation*} \begin{split}
\frac{\mathrm{d}}{\mathrm{d}t}\sqrt{c^{\epsilon}} + u^{\epsilon}\cdot\nabla\sqrt{c^{\epsilon}} = (\sqrt{c^{\epsilon}})^{-1}|\nabla\sqrt{c^{\epsilon}}|^2 + \Delta\sqrt{c^{\epsilon}} - \frac{1}{2}(n^{\epsilon}*\rho^{\epsilon})\sqrt{c^{\epsilon}}.
\end{split} \end{equation*}
Multiplying both sides of the above equation by $-4\Delta\sqrt{c^{\epsilon}}$ and integrating over $\mathbb{R}^2$, we deduce
\begin{equation} \begin{split}\label{lem3-2}
&2 \frac{\mathrm{d}}{\mathrm{d}t}\|\nabla\sqrt{c^{\epsilon}}\|_{L^2}^2 \\
&= \frac{1}{2}\int_{\mathbb{R}^2}(c^{\epsilon}(x))^{-2}|\nabla c^{\epsilon}(x)|^2\Delta c^{\epsilon}(x)\,\mathrm{d}x - \int_{\mathbb{R}^2}(c^{\epsilon}(x))^{-1}|\Delta c^{\epsilon}(x)|^2\,\mathrm{d}x \\
&\quad - \frac{1}{2}\int_{\mathbb{R}^2}(c^{\epsilon}(x))^{-1}n^{\epsilon}*\rho^{\epsilon}(x)|\nabla c^{\epsilon}(x)|^2\,\mathrm{d}x - \int_{\mathbb{R}^2}\nabla(n^{\epsilon}*\rho^{\epsilon}(x))\cdot\nabla c^{\epsilon}(x)\,\mathrm{d}x \\
&\quad + 4\int_{\mathbb{R}^2}\Delta \sqrt{c^{\epsilon}(x)}(u^{\epsilon}(x)\cdot\nabla \sqrt{c^{\epsilon}(x)})\,\mathrm{d}x.
\end{split} \end{equation}
We now show that the first two integrals on the right-hand side of \eqref{lem3-2} satisfy
\begin{equation} \begin{split}\label{www}
&\frac{1}{2}\int_{\mathbb{R}^2}(c^{\epsilon}(x))^{-2}|\nabla c^{\epsilon}(x)|^2\Delta c^{\epsilon}(x)\,\mathrm{d}x - \int_{\mathbb{R}^2}(c^{\epsilon}(x))^{-1}|\Delta c^{\epsilon}(x)|^2\,\mathrm{d}x \\
&= -\int_{\mathbb{R}^2}c^{\epsilon}(x)|D^2 \ln c^{\epsilon}(x)|^2\,\mathrm{d}x.
\end{split} \end{equation}
Indeed, integrating by parts yields
\begin{equation} \begin{split}\label{ww1}
&\int_{\mathbb{R}^2}(c^{\epsilon}(x))^{-1}\nabla c^{\epsilon}(x)\cdot\nabla\Delta c^{\epsilon}(x)\,\mathrm{d}x \\
&= \int_{\mathbb{R}^2}(c^{\epsilon}(x))^{-2}|\nabla c^{\epsilon}(x)|^2\Delta c^{\epsilon}(x)\,\mathrm{d}x - \int_{\mathbb{R}^2}(c^{\epsilon}(x))^{-1}|\Delta c^{\epsilon}(x)|^2\,\mathrm{d}x.
\end{split} \end{equation}
and
\begin{equation} \begin{split}\label{ww2}
&2\int_{\mathbb{R}^2}(c^{\epsilon}(x))^{-1}\nabla c^{\epsilon}(x)\cdot\nabla\Delta c^{\epsilon}(x)\,\mathrm{d}x = 2\int_{\mathbb{R}^2}(c^{\epsilon}(x))^{-1}\nabla c^{\epsilon}(x)\cdot \text{div}(D^2 c^{\epsilon}(x))\,\mathrm{d}x \\
&= -2\int_{\mathbb{R}^2}(c^{\epsilon}(x))^{-1}|D^2 c^{\epsilon}(x)|^2\,\mathrm{d}x + 2\int_{\mathbb{R}^2}(c^{\epsilon}(x))^{-2}(D^2 c^{\epsilon}(x)\cdot\nabla c^{\epsilon}(x))\cdot\nabla c^{\epsilon}(x)\,\mathrm{d}x.
\end{split} \end{equation}
Moreover, direct calculation gives
\begin{equation} \begin{split}\label{ww3}
&\int_{\mathbb{R}^2}(c^{\epsilon}(x))^{-2}|\nabla c^{\epsilon}(x)|^2\Delta c^{\epsilon}(x)\,\mathrm{d}x \\
&= -\int_{\mathbb{R}^2}\nabla((c^{\epsilon}(x))^{-2}|\nabla c^{\epsilon}(x)|^2)\cdot\nabla c^{\epsilon}(x)\,\mathrm{d}x \\
&= 2\int_{\mathbb{R}^2}(c^{\epsilon}(x))^{-2}|\nabla c^{\epsilon}(x)|^4\,\mathrm{d}x - 2\int_{\mathbb{R}^2}(c^{\epsilon}(x))^{-2}(D^2 c^{\epsilon}(x)\cdot\nabla c^{\epsilon}(x))\cdot\nabla c^{\epsilon}(x)\,\mathrm{d}x.
\end{split} \end{equation}
Through the algebraic operations $\eqref{ww2}+\eqref{ww3}-2\times\eqref{ww1}$, we obtain
\begin{equation} \begin{split}\label{ww4}
&\int_{\mathbb{R}^2}(c^{\epsilon}(x))^{-2}|\nabla c^{\epsilon}(x)|^2\Delta c^{\epsilon}(x)\,\mathrm{d}x \\
&= -\frac{2}{3}\int_{\mathbb{R}^2}(c^{\epsilon}(x))^{-1}|D^2 c^{\epsilon}(x)|^2\,\mathrm{d}x + \frac{2}{3}\int_{\mathbb{R}^2}(c^{\epsilon}(x))^{-3}|\nabla c^{\epsilon}(x)|^4\,\mathrm{d}x + \frac{2}{3}\int_{\mathbb{R}^2}(c^{\epsilon}(x))^{-1}|\Delta c^{\epsilon}(x)|^2\,\mathrm{d}x.
\end{split} \end{equation}
On the other hand, a direct computation shows
\begin{equation*} \begin{split}
&\int_{\mathbb{R}^2}c^{\epsilon}(x)|D^2 \ln c^{\epsilon}(x)|^2\,\mathrm{d}x \\
&= \int_{\mathbb{R}^2}(c^{\epsilon}(x))^{-1}|D^2 c^{\epsilon}(x)|^2\,\mathrm{d}x - 2\int_{\mathbb{R}^2}(c^{\epsilon}(x))^{-2}(D^2 c^{\epsilon}(x)\cdot\nabla c^{\epsilon}(x))\cdot\nabla c^{\epsilon}(x)\,\mathrm{d}x \\
&\quad + \int_{\mathbb{R}^2}(c^{\epsilon}(x))^{-3}|\nabla c^{\epsilon}(x)|^4\,\mathrm{d}x,
\end{split} \end{equation*}
which, together with \eqref{ww3}, yields
\begin{equation} \begin{split}\label{ww5}
&\int_{\mathbb{R}^2}c^{\epsilon}(x)|D^2 \ln c^{\epsilon}(x)|^2\,\mathrm{d}x \\
&= \int_{\mathbb{R}^2}(c^{\epsilon}(x))^{-1}|D^2 c^{\epsilon}(x)|^2\,\mathrm{d}x - \int_{\mathbb{R}^2}(c^{\epsilon}(x))^{-3}|\nabla c^{\epsilon}(x)|^4\,\mathrm{d}x + \int_{\mathbb{R}^2}(c^{\epsilon}(x))^{-2}|\nabla c^{\epsilon}(x)|^2\Delta c^{\epsilon}(x)\,\mathrm{d}x.
\end{split} \end{equation}
The identity \eqref{www} now follows by computing $\frac{3}{2}\times$\eqref{ww4} $+\,\,\eqref{ww5}$. As a result, equation \eqref{lem3-2} can be reformulated as
\begin{equation} \begin{split}\label{ww6}
&2 \frac{\mathrm{d}}{\mathrm{d}t}\|\nabla\sqrt{c^{\epsilon}}\|_{L^2}^2 + \int_{\mathbb{R}^2}c^{\epsilon}(x)|D^2 \ln c^{\epsilon}(x)|^2\,\mathrm{d}x \\
&= -\frac{1}{2}\int_{\mathbb{R}^2}(c^{\epsilon}(x))^{-1}n^{\epsilon}*\rho^{\epsilon}(x)|\nabla c^{\epsilon}(x)|^2\,\mathrm{d}x - \int_{\mathbb{R}^2}\nabla(n^{\epsilon}*\rho^{\epsilon}(x))\cdot\nabla c^{\epsilon}(x)\,\mathrm{d}x \\
&\quad + 4\int_{\mathbb{R}^2}\Delta \sqrt{c^{\epsilon}(x)}(u^{\epsilon}(x)\cdot\nabla \sqrt{c^{\epsilon}(x)})\,\mathrm{d}x.
\end{split} \end{equation}

Adding $\eqref{ww6}$ to \eqref{lem3-1} and noting the identity $(\nabla (n^{\epsilon}*\rho^{\epsilon}),\nabla c^{\epsilon})_{L^2}=(\nabla n^{\epsilon},\nabla (c^{\epsilon}*\rho^{\epsilon}))_{L^2}$, we obtain
\begin{equation} \begin{split}\label{lem3-6}
& \frac{\mathrm{d}}{\mathrm{d}t}\|(n^\epsilon+1)\ln(n^\epsilon+1)\|_{L^1}+2 \frac{\mathrm{d}}{\mathrm{d}t}\|\nabla\sqrt{c^{\epsilon}}\|_{L^2}^2+4\|\nabla\sqrt{n^{\epsilon}+1}\|^2_{L^2}\\
&\quad+\int_{\mathbb{R}^2}c^{\epsilon}(x)|D^2 \ln c^{\epsilon}(x)|^2\mathrm{d}x+\frac{1}{2}\int_{\mathbb{R}^2}(c^{\epsilon}(x))^{-1}(n^{\epsilon}*\rho^{\epsilon})(x)|\nabla c^{\epsilon}(x)|^2\mathrm{d}x\\
&=4\int_{\mathbb{R}^2}\Delta \sqrt{c^{\epsilon}(x)}(u^{\epsilon}(x)\cdot\nabla \sqrt{c^{\epsilon}(x)})\mathrm{d}x+\int_{\mathbb{R}^2}\Delta (c^{\epsilon}*\rho^{\epsilon}(x))\ln(n^{\epsilon}(x)+1)\mathrm{d}x\\
&:=A_1+A_2.
\end{split} \end{equation}
To estimate the terms $A_1$ and $A_2$, we claim that
\begin{equation} \begin{split}\label{ww7}
\int_{\mathbb{R}^2}(c^{\epsilon}(x))^{-3}|\nabla c^{\epsilon}(x)|^4\mathrm{d}x+\int_{\mathbb{R}^2}(c^{\epsilon}(x))^{-1}|D^2 c^{\epsilon}(x)|^2\mathrm{d}x\lesssim\int_{\mathbb{R}^2}c^{\epsilon}(x)|D^2 \ln c^{\epsilon}(x)|^2\mathrm{d}x.
\end{split} \end{equation}
Indeed, by integration by parts and H\"{o}lder's inequality, we compute
\begin{equation*} \begin{split}
&\int_{\mathbb{R}^2}(c^{\epsilon}(x))^{-3}|\nabla c^{\epsilon}(x)|^4\mathrm{d}x=\int_{\mathbb{R}^2}|\nabla \ln c^{\epsilon}(x)|^2\nabla \ln c^{\epsilon}(x)\cdot \nabla c^{\epsilon}(x)\mathrm{d}x\\
&=-\int_{\mathbb{R}^2}c^{\epsilon}(x)\nabla|\nabla \ln c^{\epsilon}(x)|^2\cdot\nabla \ln c^{\epsilon}(x) \mathrm{d}x-\int_{\mathbb{R}^2}c^{\epsilon}(x)|\nabla \ln c^{\epsilon}(x)|^2\Delta \ln c^{\epsilon}(x) \mathrm{d}x\\
&=-2\int_{\mathbb{R}^2}(c^{\epsilon}(x))^{-1}(D^2\ln c^{\epsilon}(x)\cdot\nabla c^{\epsilon}(x))\cdot\nabla c^{\epsilon}(x) \mathrm{d}x\\
&\quad-\int_{\mathbb{R}^2}(c^{\epsilon}(x))^{-1}|\nabla c^{\epsilon}(x)|^2\Delta \ln c^{\epsilon}(x) \mathrm{d}x\\
&\leq2\|(c^{\epsilon})^{\frac{1}{2}}|D^2\ln c^{\epsilon}|\|_{L^2}\|(c^{\epsilon})^{-\frac{3}{2}}|\nabla c^{\epsilon}|^2\|_{L^2}+\|(c^{\epsilon})^{\frac{1}{2}}|\Delta\ln c^{\epsilon}|\|_{L^2}\|(c^{\epsilon})^{-\frac{3}{2}}|\nabla c^{\epsilon}|^2\|_{L^2}\\
&\leq5\|(c^{\epsilon})^{\frac{1}{2}}|D^2\ln c^{\epsilon}|\|_{L^2}\|(c^{\epsilon})^{-\frac{3}{2}}|\nabla c^{\epsilon}|^2\|_{L^2},
\end{split} \end{equation*}
which implies
\begin{equation} \begin{split}\label{ww8}
\int_{\mathbb{R}^2}(c^{\epsilon}(x))^{-3}|\nabla c^{\epsilon}(x)|^4\mathrm{d}x
\leq25\|(c^{\epsilon})^{\frac{1}{2}}|D^2\ln c^{\epsilon}|\|_{L^2}^2.
\end{split} \end{equation}
Moreover, using the inequality $(f+g)^2\geq\frac{1}{2}f^2-g^2$, we deduce
\begin{equation*} \begin{split}
c^{\epsilon}|\partial_{ij}\ln c^{\epsilon}|^2&=c^{\epsilon}|(c^{\epsilon})^{-1}\partial_{ij} c^{\epsilon}-(c^{\epsilon})^{-2}\partial_{i} c^{\epsilon}\partial_{j} c^{\epsilon}|^2\\
&\geq\frac{1}{2}c^{\epsilon}|(c^{\epsilon})^{-1}\partial_{ij} c^{\epsilon}|^2-c^{\epsilon}|(c^{\epsilon})^{-2}\partial_{i} c^{\epsilon}\partial_{j} c^{\epsilon}|^2\\
&=\frac{1}{2}(c^{\epsilon})^{-1}|\partial_{ij} c^{\epsilon}|^2-(c^{\epsilon})^{-3}|\partial_{i} c^{\epsilon}\partial_{j} c^{\epsilon}|^2,\quad 1\leq i,~j\leq2,
\end{split} \end{equation*}
which combined with \eqref{ww8} gives
\begin{equation*} \begin{split}
\int_{\mathbb{R}^2}(c^{\epsilon}(x))^{-1}|D^2 c^{\epsilon}(x)|^2\mathrm{d}x&\lesssim\int_{\mathbb{R}^2}(c^{\epsilon}(x))^{-3}|\nabla c^{\epsilon}(x)|^4\mathrm{d}x+\int_{\mathbb{R}^2}c^{\epsilon}(x)|D^2 \ln c^{\epsilon}(x)|^2\mathrm{d}x\\
&\lesssim\int_{\mathbb{R}^2}c^{\epsilon}(x)|D^2 \ln c^{\epsilon}(x)|^2\mathrm{d}x.
\end{split} \end{equation*}
This proves the inequality \eqref{ww7}. Now, we proceed to estimate the terms $A_1$ and $A_2$ in \eqref{lem3-6}. By applying the divergence-free condition $\nabla\cdot u^{\epsilon}=0$ and Young's inequality, it follows from \eqref{ww7} that for any $\eta>0$,
\begin{equation} \begin{split}\label{ww9}
A_1&=-4\int_{\mathbb{R}^2}(\nabla\sqrt{c^{\epsilon}(x)}\cdot\nabla)u^{\epsilon}(x)\cdot\nabla \sqrt{c^{\epsilon}(x)}\mathrm{d}x\\
&\leq4\int_{\mathbb{R}^2}\sqrt{c^{\epsilon}(x)}|\nabla u^{\epsilon}(x)|(c^{\epsilon}(x))^{-\frac{1}{2}}|\nabla\sqrt{c^{\epsilon}(x)}|^2\mathrm{d}x\\
&\leq\eta\int_{\mathbb{R}^2}(c^{\epsilon}(x))^{-3}|\nabla c^{\epsilon}(x)|^4\mathrm{d}x+C_{\eta}\|c_0\|_{L^{\infty}}\|\nabla u^{\epsilon}\|_{L^2}^2\\
&\leq C_1\eta\int_{\mathbb{R}^2}c^{\epsilon}(x)|D^2 \ln c^{\epsilon}(x)|^2\mathrm{d}x+C_{\eta,\|c_0\|_{L^{\infty}}}\|\nabla u^{\epsilon}\|_{L^2}^2.
\end{split} \end{equation}
Noting that
\begin{equation*} \begin{split}
\|\Delta (c^{\epsilon}*\rho^{\epsilon})\|_{L^2}\leq \|\Delta c^{\epsilon}\|_{L^2},\quad\Delta c^{\epsilon}=2|\nabla\sqrt{c^{\epsilon}}|^2+2\sqrt{c^{\epsilon}}\Delta\sqrt{c^{\epsilon}},
\end{split} \end{equation*}
we infer from Young's inequality and \eqref{ww7} that
\begin{equation} \begin{split}\label{lem3-8}
A_2&\leq\frac{\eta}{\|c_0\|_{L^{\infty}}}\|\Delta c^{\epsilon}\|_{L^2}^2+C_{\eta,\|c_0\|_{L^{\infty}}}\|\ln(n^{\epsilon}+1)\|_{L^2}^2\\
&\leq\frac{8\eta}{\|c_0\|_{L^{\infty}}}\int_{\mathbb{R}^2}(|\nabla\sqrt{c^{\epsilon}(x)}|^4+c^{\epsilon}(x)|\Delta\sqrt{c^{\epsilon}(x)}|^2)\mathrm{d}x\\
&\quad+C_{\eta,\|c_0\|_{L^{\infty}}}\int_{\mathbb{R}^2}(n^{\epsilon}(x)+1)\ln(n^{\epsilon}(x)+1)\mathrm{d}x\\
&\leq8\eta\|(\sqrt{c^{\epsilon}})^{-1}|\nabla\sqrt{c^{\epsilon}}|^2\|_{L^2}^2
+8\eta\|\Delta\sqrt{c^{\epsilon}}\|_{L^2}^2+C_{\eta,\|c_0\|_{L^{\infty}}}\|(n^{\epsilon}+1)\ln(n^{\epsilon}+1)\|_{L^1}\\
&\leq C\eta\|(c^{\epsilon})^{-\frac{3}{2}}|\nabla c^{\epsilon}|^2\|_{L^2}^2+C\eta\|(c^{\epsilon})^{-\frac{1}{2}}|D^2 c^{\epsilon}|\|_{L^2}^2+C_{\eta,\|c_0\|_{L^{\infty}}}\|(n^{\epsilon}+1)\ln(n^{\epsilon}+1)\|_{L^1}\\
&\leq C_2\eta\int_{\mathbb{R}^2}c^{\epsilon}(x)|D^2 \ln c^{\epsilon}(x)|^2\mathrm{d}x+C_{\eta,\|c_0\|_{L^{\infty}}}\|(n^{\epsilon}+1)\ln(n^{\epsilon}+1)\|_{L^1}.
\end{split} \end{equation}
Plugging the estimates \eqref{ww9} and \eqref{lem3-8} into \eqref{lem3-6} and choosing $\eta > 0$ small enough such that $\eta \leq \frac{1}{2(C_1 + C_2)}$, we obtain
\begin{equation} \begin{split}\label{lem3-9}
& \frac{\mathrm{d}}{\mathrm{d}t}\|(n^\epsilon+1)\ln(n^\epsilon+1)\|_{L^1} + 2 \frac{\mathrm{d}}{\mathrm{d}t}\|\nabla\sqrt{c^{\epsilon}}\|_{L^2}^2 + 4\|\nabla\sqrt{n^{\epsilon}+1}\|^2_{L^2} \\
&\quad + \frac{1}{2} \int_{\mathbb{R}^2} c^{\epsilon}(x)|D^2 \ln c^{\epsilon}(x)|^2\,\mathrm{d}x + \frac{1}{2} \int_{\mathbb{R}^2} (c^{\epsilon}(x))^{-1} (n^{\epsilon}*\rho^{\epsilon})(x) |\nabla c^{\epsilon}(x)|^2\,\mathrm{d}x \\
&\leq C_{\|c_0\|_{L^{\infty}}} \|\nabla u^{\epsilon}\|_{L^2}^2 + C_{\|c_0\|_{L^{\infty}}} \|(n^{\epsilon}+1)\ln(n^{\epsilon}+1)\|_{L^1}.
\end{split} \end{equation}
Moreover, we have
\begin{equation} \begin{split} \label{lem3-12}
\|u^{\epsilon}(t)\|_{L^2}^2 + \int_0^t \|\nabla u^{\epsilon}(r)\|_{L^2}^2\,\mathrm{d}r
\leq \|u^{\epsilon}(0)\|_{L^2}^2 + \int_0^t \|\nabla\phi\|_{L^{\infty}} \|n^{\epsilon}(r)\|_{L^2} \|u^{\epsilon}(r)\|_{L^2}\,\mathrm{d}r.
\end{split} \end{equation}
To estimate $\|n^{\epsilon}\|_{L^2}$, we note that
\begin{equation} \begin{split} \label{abc}
\|n^{\epsilon}\|_{L^2}
&\lesssim \|n^{\epsilon} \mathbf{1}_{\{0<n^{\epsilon}<1\}}\|_{L^2} + \|n^{\epsilon} \mathbf{1}_{\{1 \leq n^{\epsilon}\}}\|_{L^2} \\
&\lesssim \|n^{\epsilon}\|_{L^1}^{1/2} + \left\|(n^{\epsilon})^{1/2} \mathbf{1}_{\{1 \leq n^{\epsilon}\}}\right\|_{L^4}^2 \\
&\lesssim \|n^{\epsilon}\|_{L^1}^{1/2} + \|\nabla((n^{\epsilon})^{1/2} \mathbf{1}_{\{1 \leq n^{\epsilon}\}})\|_{L^2} \cdot \left\|(n^{\epsilon})^{1/2} \mathbf{1}_{\{1 \leq n^{\epsilon}\}}\right\|_{L^2} \\
&\lesssim \left(1 + \|\nabla\sqrt{n^{\epsilon}+1}\|_{L^2} \right) \|n^{\epsilon}\|_{L^1}^{1/2}.
\end{split} \end{equation}
Using Young's inequality, we infer from \eqref{lem3-12} and \eqref{abc} that
\begin{equation} \begin{split} \label{lem3-122}
\|u^{\epsilon}(t)\|_{L^2}^2 + \int_0^t \|\nabla u^{\epsilon}(r)\|_{L^2}^2\,\mathrm{d}r
&\leq \|u^{\epsilon}(0)\|_{L^2}^2 + \eta t + \eta \int_0^t \|\nabla\sqrt{n^{\epsilon}(r)+1}\|_{L^2}^2\,\mathrm{d}r \\
&\quad + C_{\eta, \|n_0\|_{L^1}, \|\nabla\phi\|_{L^{\infty}}} \int_0^t \|u^{\epsilon}(r)\|_{L^2}^2\,\mathrm{d}r.
\end{split} \end{equation}
Multiplying both sides of \eqref{lem3-122} by $C_{\|c_0\|_{L^{\infty}}}$ and choosing $\eta \ll \frac{1}{C_{\|c_0\|_{L^{\infty}}}}$, we deduce from \eqref{lem3-9} and the equivalence
$
\|(n^\epsilon+1) \ln(n^\epsilon+1)\|_{L^1} \asymp \|n^\epsilon\|_{L^1 \cap L\ln L} = \|n^\epsilon\|_{L^1} + \|n^\epsilon\|_{L \ln L}
$
that
\begin{equation*} \begin{split}
\mathcal{F}_1(n^{\epsilon}, c^{\epsilon}, u^{\epsilon})(t) + \int_0^t \mathcal{G}_1(n^{\epsilon}, c^{\epsilon}, u^{\epsilon})(r)\,\mathrm{d}r
\lesssim_{n_0, c_0, u_0, \phi} \mathcal{F}_1(n^{\epsilon}, c^{\epsilon}, u^{\epsilon})(0) + t + \int_0^t \mathcal{F}_1(n^{\epsilon}, c^{\epsilon}, u^{\epsilon})(r)\,\mathrm{d}r.
\end{split} \end{equation*}
The desired estimate \eqref{5lem-2} then follows from Gronwall's inequality. Moreover, by using the estimates \eqref{5lem-2} and \eqref{ww7}, it follows that
\begin{equation*} \begin{split}
\sup_{r\in[0,t]} \|\nabla c^{\epsilon}(r)\|_{L^2}^{2}
&\lesssim \sup_{r\in[0,t]} \|\sqrt{c^{\epsilon}(r)}\|_{L^{\infty}}^2 \sup_{r\in[0,t]} \|\nabla\sqrt{c^{\epsilon}(r)}\|_{L^2}^{2} \\
&\lesssim \|c_0\|_{L^{\infty}} \sup_{r\in[0,t]} \|\nabla\sqrt{c^{\epsilon}(r)}\|_{L^2}^{2} \\
&\lesssim_{n_0, c_0, u_0, \phi, t} 1,
\end{split} \end{equation*}
and
\begin{equation*} \begin{split}
\int_0^t \|\Delta c^{\epsilon}(r)\|_{L^2}^2\,\mathrm{d}r
&\lesssim \|c_0\|_{L^{\infty}} \int_0^t \left( \|(\sqrt{c^{\epsilon}(r)})^{-1} |\nabla\sqrt{c^{\epsilon}(r)}|^2\|_{L^2}^2 + \|\Delta\sqrt{c^{\epsilon}(r)}\|_{L^2}^2 \right)\,\mathrm{d}r \\
&\lesssim \|c_0\|_{L^{\infty}} \int_0^t \|\sqrt{c^{\epsilon}(r)} |D^2 \ln c^{\epsilon}(r)|\|_{L^2}^2\,\mathrm{d}r \\
&\lesssim_{n_0, c_0, u_0, \phi, t} 1,
\end{split} \end{equation*}
which implies the estimate \eqref{5lem-3}.
Furthermore, we derive from the $n^{\epsilon}$-equation that
\begin{equation*} \begin{split}
\frac{1}{2} \frac{\mathrm{d}}{\mathrm{d}t} \|n^{\epsilon}\|_{L^2}^2 + \|\nabla n^{\epsilon}\|_{L^2}^2
&= (n^{\epsilon} \nabla (c^{\epsilon} * \rho^{\epsilon}), \nabla n^{\epsilon})_{L^2} \\
&\leq \|n^{\epsilon}\|_{L^4} \|\nabla c^{\epsilon}\|_{L^4} \|\nabla n^{\epsilon}\|_{L^2} \\
&\leq \frac{1}{4} \|\nabla n^{\epsilon}\|_{L^2}^2 + C \|n^{\epsilon}\|_{L^2} \|\nabla n^{\epsilon}\|_{L^2} \|\nabla c^{\epsilon}\|_{L^2} \|\Delta c^{\epsilon}\|_{L^2} \\
&\leq \frac{1}{2} \|\nabla n^{\epsilon}\|_{L^2}^2 + C \|\nabla c^{\epsilon}\|_{L^2}^2 \|\Delta c^{\epsilon}\|_{L^2}^2 \|n^{\epsilon}\|_{L^2}^2.
\end{split} \end{equation*}
By using the Gronwall lemma, we derive from \eqref{5lem-3} that \eqref{5lem-4} holds. The proof is thus completed.
\end{proof}

Based on the lower-order uniform bounded estimates provided by Lemma \ref{5lem}, we can further enhance the time-space regularity of the solution $(n^{\epsilon},c^{\epsilon},u^{\epsilon})$ through energy estimates.
\begin{lemma} \label{5pro}
Let the assumptions in Theorem \ref{th1} be satisfied. Let $(n^\epsilon,c^\epsilon,u^\epsilon)$ be the global solution of \eqref{Mod-1} with $\epsilon \in (0,1)$. Then the following estimates hold:
\begin{align}
&\|u^{\epsilon}(t)\|_{H^1}^2+\int_0^t\|u^{\epsilon}(r)\|_{H^2}^2\mathrm{d}r\lesssim_{n_0,c_0,u_0,\phi,t} 1,\label{5pro-1}\\
&\|c^{\epsilon}(t)\|_{H^2}^2+\int_0^t\|c^{\epsilon}(r)\|_{H^3}^2\mathrm{d}r\lesssim_{n_0,c_0,u_0,\phi,t} 1,\label{5pro-2}\\
&\|n^{\epsilon}(t)\|_{H^1}^2+\int_0^t\|n^{\epsilon}(r)\|_{H^2}^2\mathrm{d}r\lesssim_{n_0,c_0,u_0,\phi,t} 1.\label{5pro-3}
\end{align}
\end{lemma}

\begin{proof}[\emph{\textbf{Proof}}]
Multiplying the third equation of \eqref{Mod-1} by $-\Delta u^{\epsilon}$ and applying the interpolation inequality, we obtain
\begin{equation}\label{5pro-4}
\begin{split}
\frac{1}{2}\frac{\mathrm{d}}{\mathrm{d}t}\|\nabla u^{\epsilon}\|_{L^2}^2+\|\Delta u^{\epsilon}\|_{L^2}^2 &= ((u^{\epsilon}\cdot \nabla) u^{\epsilon},\Delta u^{\epsilon})_{L^2} - ((n^{\epsilon}\nabla \phi)*\rho^\epsilon,\Delta u^{\epsilon})_{L^2}\\
&\leq \frac{1}{2}\|\Delta u^{\epsilon}\|_{L^2}^2 + \|u^{\epsilon}\|_{L^4}^2\|\nabla u^{\epsilon}\|_{L^4}^2 + \|\nabla \phi\|_{L^{\infty}}^2\|n^{\epsilon}\|_{L^2}^2\\
&\leq \frac{3}{4}\|\Delta u^{\epsilon}\|_{L^2}^2 + \|\nabla \phi\|_{L^{\infty}}^2\|n^{\epsilon}\|_{L^2}^2 + C\|u^{\epsilon}\|_{L^2}^2\|\nabla u^{\epsilon}\|_{L^2}^4.
\end{split}
\end{equation}
Since by \eqref{5lem-2}, we have $\int_0^t\|u^{\epsilon}(r)\|_{L^2}^2\|\nabla u^{\epsilon}(r)\|_{L^2}^2\mathrm{d}r<\infty$, it follows from \eqref{5lem-2}, \eqref{5lem-4}, \eqref{5pro-4} and Gronwall's lemma that \eqref{5pro-1} holds.

To derive \eqref{5pro-2}, we multiply the second equation of \eqref{Mod-1} by $-\Delta^2 c^{\epsilon}$ and use the interpolation inequality to get
\begin{equation*}
\begin{split}
\frac{1}{2}\frac{\mathrm{d}}{\mathrm{d}t}\|\Delta c^{\epsilon}\|_{L^2}^2 + \|\nabla\Delta c^{\epsilon}\|_{L^2}^2 &= (\nabla(u^{\epsilon}\cdot\nabla c^{\epsilon}),\nabla\Delta c^{\epsilon})_{L^2} + (\nabla(c^{\epsilon}n^{\epsilon}*\rho^{\epsilon}),\nabla\Delta c^{\epsilon})_{L^2}\\
&\leq \frac{1}{2}\|\nabla\Delta c^{\epsilon}\|_{L^2}^2 + \|\nabla u^{\epsilon}\|_{L^4}^2\|\nabla c^{\epsilon}\|_{L^4}^2 + \|u^{\epsilon}\|_{L^{\infty}}^2\|D^2 c^{\epsilon}\|_{L^2}^2\\
&\quad + \|\nabla c^{\epsilon}\|_{L^4}^2\|n^{\epsilon}\|_{L^4}^2 + \|c^{\epsilon}\|_{L^{\infty}}^2\|\nabla n^{\epsilon}\|_{L^2}^2\\
&\leq \frac{1}{2}\|\nabla\Delta c^{\epsilon}\|_{L^2}^2 + C(\|n^{\epsilon}\|_{L^2}^2 + \|c^{\epsilon}\|_{L^{\infty}}^2)\|\nabla n^{\epsilon}\|_{L^2}^2\\
&\quad + C\|\nabla u^{\epsilon}\|_{L^2}^2\|\Delta u^{\epsilon}\|_{L^2}^2 + C(\|\nabla c^{\epsilon}\|_{L^2}^2 + \|u^{\epsilon}\|_{H^2}^2)\|\Delta c^{\epsilon}\|_{L^2}^2.
\end{split}
\end{equation*}
Combining the above with \eqref{5lem-3}, \eqref{5lem-4}, \eqref{5pro-1}, and applying Gronwall's lemma yields \eqref{5pro-2}.

Lastly, we establish \eqref{5pro-3}. Multiplying the first equation of \eqref{Mod-1} by $-\Delta n^{\epsilon}$ and integrating over space, we deduce via interpolation that
\begin{equation*}
\begin{split}
\frac{1}{2}\frac{\mathrm{d}}{\mathrm{d}t}\|\nabla n^{\epsilon}\|_{L^2}^2 + \|\Delta n^{\epsilon}\|_{L^2}^2 &= (u^{\epsilon}\cdot\nabla n^{\epsilon},\Delta n^{\epsilon})_{L^2} + (\nabla\cdot(n^{\epsilon}\nabla (c^{\epsilon}*\rho^{\epsilon})),\Delta n^{\epsilon})_{L^2}\\
&\leq \frac{1}{2}\|\Delta n^{\epsilon}\|_{L^2}^2 + \|u^{\epsilon}\cdot\nabla n^{\epsilon}\|_{L^2}^2 + \|\nabla\cdot(n^{\epsilon}\nabla (c^{\epsilon}*\rho^{\epsilon}))\|_{L^2}^2\\
&\leq \frac{1}{2}\|\Delta n^{\epsilon}\|_{L^2}^2 + (\|u^{\epsilon}\|_{L^{\infty}}^2 + \|\nabla c^{\epsilon}\|_{L^{\infty}}^2)\|\nabla n^{\epsilon}\|_{L^2}^2 + \|n^{\epsilon}\|_{L^4}^2\|\Delta c^{\epsilon}\|_{L^4}^2\\
&\leq \frac{1}{2}\|\Delta n^{\epsilon}\|_{L^2}^2 + C\|\Delta c^{\epsilon}\|_{L^2}^2\|\nabla\Delta c^{\epsilon}\|_{L^2}^2\\
&\quad + C(\|u^{\epsilon}\|_{H^2}^2 + \|c^{\epsilon}\|_{H^3}^2 + \|n^{\epsilon}\|_{L^2}^2)\|\nabla n^{\epsilon}\|_{L^2}^2.
\end{split}
\end{equation*}
Using \eqref{5lem-4}, \eqref{5pro-1}, \eqref{5pro-2}, and applying Gronwall's lemma gives the desired bound \eqref{5pro-3}. This completes the proof.
\end{proof}

\begin{corollary} \label{cor1}
Let $(n^\epsilon,c^\epsilon,u^\epsilon)$ be the global solution of \eqref{Mod-1} constructed in Lemma \ref{lem2} with $\epsilon \in (0,1)$. Then we have
\begin{align}
&\int_0^t\left\|\frac{\partial n^{\epsilon}}{\partial r}\right\|_{L^{2}}^2\mathrm{d}r \lesssim_{n_0,c_0,u_0,\phi,t} 1,\label{cor1-1}\\
&\int_0^t\left\|\frac{\partial c^{\epsilon}}{\partial r}\right\|_{H^{1}}^2\mathrm{d}r \lesssim_{n_0,c_0,u_0,\phi,t} 1,\label{cor1-2}\\
&\int_0^t\left\|\frac{\partial u^{\epsilon}}{\partial r}\right\|_{L^{2}}^2\mathrm{d}r \lesssim_{n_0,c_0,u_0,\phi,t} 1.\label{cor1-3}
\end{align}
\end{corollary}

\begin{proof}[\emph{\textbf{Proof}}]
For \eqref{cor1-1}, we derive from \eqref{5pro-1}-\eqref{5pro-3} that
\begin{equation*}
\begin{split}
\int_0^t\left\|\frac{\partial n^{\epsilon}}{\partial r}\right\|_{L^{2}}^2\mathrm{d}r
&\lesssim \int_0^t\Big( \|u^\epsilon(r)\|_{L^{\infty}}^2\|\nabla n^{\epsilon}(r)\|_{L^2}^2 + \|\Delta n^\epsilon(r)\|_{L^2}^2 + \|\nabla c^\epsilon(r)\|_{L^{\infty}}^2\|\nabla n^{\epsilon}(r)\|_{L^2}^2 \\
&\quad + \|n^\epsilon(r)\|_{L^{\infty}}^2\|\Delta c^{\epsilon}(r)\|_{L^2}^2 \Big) \mathrm{d}r \\
&\lesssim \int_0^t\Big( \|u^\epsilon(r)\|_{H^{2}}^2\|n^{\epsilon}(r)\|_{H^1}^2 + \|n^\epsilon(r)\|_{H^2}^2 + \|c^\epsilon(r)\|_{H^3}^2\|n^{\epsilon}(r)\|_{H^1}^2 \\
&\quad + \|n^\epsilon(r)\|_{H^2}^2\|c^{\epsilon}(r)\|_{H^2}^2 \Big) \mathrm{d}r \\
&\lesssim_{n_0,c_0,u_0,\phi,t} 1.
\end{split}
\end{equation*}

For \eqref{cor1-2}, it follows from \eqref{5pro-1}-\eqref{5pro-3} that
\begin{equation*}
\begin{split}
\int_0^t\left\|\frac{\partial c^{\epsilon}}{\partial r}\right\|_{H^{1}}^2\mathrm{d}r
&= \int_0^t\left\|\frac{\partial c^{\epsilon}}{\partial r}\right\|_{L^{2}}^2\mathrm{d}r + \int_0^t\left\|\nabla\frac{\partial c^{\epsilon}}{\partial r}\right\|_{L^{2}}^2\mathrm{d}r \\
&\lesssim \int_0^t\Big( \|u^\epsilon(r)\|_{H^2}^2\|c^{\epsilon}(r)\|_{H^1}^2 + \|c^\epsilon(r)\|_{H^2}^2 + \|c^\epsilon(r)\|_{L^{\infty}}^2\|n^{\epsilon}(r)\|_{L^2}^2 \Big) \mathrm{d}r \\
&\quad + \int_0^t\Big( \|u^\epsilon(r)\|_{H^1}^2\|c^{\epsilon}(r)\|_{H^3}^2 + \|u^\epsilon(r)\|_{H^2}^2\|c^{\epsilon}(r)\|_{H^2}^2 + \|c^\epsilon(r)\|_{H^3}^2 \\
&\quad + \|c^\epsilon(r)\|_{H^3}^2\|n^{\epsilon}(r)\|_{L^2}^2 + \|c^{\epsilon}(r)\|_{L^{\infty}}^2\|n^{\epsilon}(r)\|_{H^1}^2 \Big) \mathrm{d}r \\
&\lesssim_{n_0,c_0,u_0,\phi,t} 1.
\end{split}
\end{equation*}
Similarly, it follows that
\begin{equation*}
\begin{split}
\int_0^t\left\|\frac{\partial u^{\epsilon}}{\partial r}\right\|_{L^{2}}^2\mathrm{d}r
&\lesssim \int_0^t\Big( \|u^\epsilon(r)\|_{H^2}^2\|u^{\epsilon}(r)\|_{H^1}^2 + \|u^\epsilon(r)\|_{H^2}^2 + \|\nabla\phi\|_{L^{\infty}}^2\|n^{\epsilon}(r)\|_{L^2}^2 \Big) \mathrm{d}r \\
&\lesssim_{n_0,c_0,u_0,\phi,t} 1,
\end{split}
\end{equation*}
which gives \eqref{cor1-3}. The proof of Corollary \ref{cor1} is thus completed.
\end{proof}

\subsection{Existence of global strong solutions}\label{sec3-1}
We begin by introducing the following phase spaces:
\begin{equation}\label{work}
\begin{split}
\mathcal{Z}_n &:= \mathcal{C}([0,T];U') \cap L^2_w(0,T;H^2(\mathbb{R}^2)) \cap L^2(0,T;H^1_{\mathrm{loc}}(\mathbb{R}^2)) \cap \mathcal{C}([0,T];H^1_w(\mathbb{R}^2)), \\
\mathcal{Z}_c &:= \mathcal{C}([0,T];U') \cap L^2_w(0,T;H^3(\mathbb{R}^2)) \cap L^2(0,T;H^2_{\mathrm{loc}}(\mathbb{R}^2)) \cap \mathcal{C}([0,T];H^2_w(\mathbb{R}^2)), \\
\mathcal{Z}_u &:= \mathcal{C}([0,T];U'_1) \cap L^2_w(0,T;\mathbb{H}^2) \cap L^2(0,T;\mathbb{H}^1_{\mathrm{loc}}) \cap \mathcal{C}([0,T];\mathbb{H}^1_w),
\end{split}
\end{equation}
where the Hilbert spaces $U$ and $U_1$ satisfy the property that, for any fixed $s > 0$, the embeddings $U \subset H^s(\mathbb{R}^3)$ and $U_1 \subset \mathbb{H}^s$ are compact (cf.~\cite{11brzezniak2013existence}). The symbols $U'$ and $U_1'$ denote their respective dual spaces. Based on the uniform a priori estimates provided by Lemma~\ref{5pro} and Corollary~\ref{cor1}, together with standard compactness arguments, we derive the following convergence result:
\begin{equation}\label{5-1}
\begin{split}
n^{\epsilon} \to n \ \text{in}\ \mathcal{Z}_n,\quad c^{\epsilon} \to c \ \text{in}\ \mathcal{Z}_c,\quad u^{\epsilon} \to u \ \text{in}\ \mathcal{Z}_u,\quad \text{as}~\epsilon \to 0.
\end{split}
\end{equation}
In fact, the convergence \eqref{5-1} is sufficient to guarantee that the limit $(n,c,u)$ constitutes a global strong solution to system~\eqref{CNS0} in the sense of Theorem~\ref{th1}. As a representative case, we shall verify the convergence of the cross-diffusion term in the $n$-equation, as the convergence of the remaining terms, including those in the $c$- and $u$-equations, can be handled analogously and is therefore omitted.
Specifically, we aim to show that for all $t \in [0,T]$,
\begin{equation}\label{ddd}
\begin{split}
\lim_{\epsilon\to 0} \int_0^t \big(\nabla \cdot (n^{\epsilon}(r)\nabla (c^{\epsilon}(r)*\rho^{\epsilon})), \varphi\big)_{L^2}\,\mathrm{d}r = \int_0^t \big(\nabla \cdot (n(r)\nabla c(r)), \varphi\big)_{L^2}\,\mathrm{d}r,\quad \forall \varphi \in L^2(\mathbb{R}^2).
\end{split}
\end{equation}
To this end, let $\varphi \in L^2(\mathbb{R}^2)$ and $h>0$ be arbitrary. Then there exists a test function $\varphi_h \in \mathcal{C}_0^{\infty}(\mathbb{R}^2)$ such that $\|\varphi - \varphi_h\|_{L^2} \leq h$. Moreover, there exists $d > 0$ such that $\mathrm{supp}\,\varphi_h \subset \mathcal{O}_d$, where $\{\mathcal{O}_d\}_{d \in \mathbb{N}}$ denotes the sequence of bounded open subsets defined in Subsection~\ref{nnnn}. Using the triangle inequality, we estimate
\begin{equation}\label{d.1}
\begin{split}
&\left| \left( \nabla \cdot (n^{\epsilon} \nabla (c^{\epsilon} * \rho^{\epsilon})) - \nabla \cdot (n \nabla c), \varphi \right)_{L^2} \right| \\
&\leq \left( \| \nabla \cdot (n^{\epsilon} \nabla (c^{\epsilon} * \rho^{\epsilon})) \|_{L^2} + \| \nabla \cdot (n \nabla c) \|_{L^2} \right) \| \varphi - \varphi_h \|_{L^2} \\
&\quad + \left| \big( \nabla \cdot ((n^{\epsilon} - n) \nabla (c^{\epsilon} * \rho^{\epsilon})), \varphi_h \big)_{L^2} \right| + \left| \big( \nabla \cdot (n \nabla (c^{\epsilon} * \rho^{\epsilon} - c)), \varphi_h \big)_{L^2} \right|.
\end{split}
\end{equation}
By Sobolev embeddings $H^2(\mathbb{R}^2) \hookrightarrow L^\infty(\mathbb{R}^2)$ and $H^1(\mathbb{R}^2) \hookrightarrow L^p(\mathbb{R}^2)$ for $p \in [2, \infty)$, we can estimate the terms on the right-hand side as
\begin{equation*}
\begin{split}
&\left( \| \nabla \cdot (n^{\epsilon} \nabla (c^{\epsilon} * \rho^{\epsilon})) \|_{L^2} + \| \nabla \cdot (n \nabla c) \|_{L^2} \right) \| \varphi - \varphi_h \|_{L^2} \\
&\lesssim \left( \|n^{\epsilon}\|_{H^1} \|c^{\epsilon}\|_{H^3} + \|n\|_{H^1} \|c\|_{H^3} \right) h.
\end{split}
\end{equation*}
For the other two terms, we have
\begin{equation*}
\left| \big( \nabla \cdot ((n^{\epsilon} - n) \nabla (c^{\epsilon} * \rho^{\epsilon})), \varphi_h \big)_{L^2} \right| \lesssim \|n^{\epsilon} - n\|_{H^1(\mathcal{O}_d)} \|c^{\epsilon}\|_{H^3(\mathcal{O}_d)} \|\varphi_h\|_{L^2(\mathcal{O}_d)},
\end{equation*}
and
\begin{equation*}
\begin{split}
\left| \big( \nabla \cdot (n \nabla (c^{\epsilon} * \rho^{\epsilon} - c)), \varphi_h \big)_{L^2} \right|
&\lesssim \|n\|_{H^2(\mathcal{O}_d)} \|c^{\epsilon} * \rho^{\epsilon} - c\|_{H^2(\mathcal{O}_d)} \|\varphi_h\|_{L^2(\mathcal{O}_d)}.
\end{split}
\end{equation*}
Substituting these estimates into \eqref{d.1}, we deduce that for all $t \in [0,T]$,
\begin{equation}\label{d.2}
\begin{split}
&\left| \int_0^t \big( \nabla \cdot (n^{\epsilon}(r) \nabla (c^{\epsilon}(r) * \rho^{\epsilon})), \varphi \big)_{L^2} \mathrm{d}r - \int_0^t \big( \nabla \cdot (n(r) \nabla c(r)), \varphi \big)_{L^2} \mathrm{d}r \right| \\
&\lesssim h \int_0^T \big( \|n^{\epsilon}(t)\|_{H^1}^2 + \|c^{\epsilon}(t)\|_{H^3}^2 + \|n(t)\|_{H^1}^2 + \|c(t)\|_{H^3}^2 \big)\mathrm{d}t \\
&\quad + \|\varphi_h\|_{L^2(\mathcal{O}_d)} \|c^{\epsilon}\|_{L^2(0,T;H^3(\mathcal{O}_d))} \|n^{\epsilon} - n\|_{L^2(0,T;H^1(\mathcal{O}_d))} \\
&\quad + \|\varphi_h\|_{L^2(\mathcal{O}_d)} \|n\|_{L^2(0,T;H^2(\mathcal{O}_d))} \|c^{\epsilon} * \rho^{\epsilon} - c\|_{L^2(0,T;H^2(\mathcal{O}_d))}.
\end{split}
\end{equation}
By taking the limit superior as $\epsilon \to 0$ and invoking the convergence result \eqref{5-1}, we obtain
\begin{equation*}
\limsup_{\epsilon \to 0} \left| \int_0^t \big( \nabla \cdot (n^{\epsilon} \nabla (c^{\epsilon} * \rho^{\epsilon})) - \nabla \cdot (n \nabla c), \varphi \big)_{L^2} \mathrm{d}r \right| \leq C h,
\end{equation*}
which, together with the arbitrariness of $h > 0$, yields the desired convergence \eqref{ddd}.

As a result, we obtain the integral identities \eqref{iden}. Furthermore, it follows from Corollary~\ref{cor1} that
\begin{equation*}
\frac{\partial n}{\partial t} \in L^2(0,T;L^2(\mathbb{R}^2)),\quad \frac{\partial c}{\partial t} \in L^2(0,T;H^1(\mathbb{R}^2)),\quad \frac{\partial u}{\partial t} \in L^2(0,T;\mathbb{H}^0).
\end{equation*}
Thus, by the Lions-Magenes lemma \cite{ad-2lions1963problemes}, together with \eqref{work} and \eqref{5-1}, we conclude that
\begin{equation}\label{reg}
\begin{split}
&n \in \mathcal{C}([0,T];H^1(\mathbb{R}^2)) \cap L^2(0,T;H^2(\mathbb{R}^2)),\\
&c \in \mathcal{C}([0,T];H^2(\mathbb{R}^2)) \cap L^2(0,T;H^3(\mathbb{R}^2)),\\
&u \in \mathcal{C}([0,T]; \mathbb{H}^1) \cap L^2(0,T; \mathbb{H}^2).
\end{split}
\end{equation}
This completes the proof of the existence of a global strong solution. \qed

\subsection{Uniqueness of the global strong solution}\label{sec3-2}
Let $(n_1,c_1,u_1)$ and $(n_2,c_2,u_2)$ be two global strong solutions to the system \eqref{CNS0} with the same initial data $(n_0,c_0,u_0)$. For simplicity, we set
\begin{equation}\label{fin1}
\begin{split}
n^* = n_1 - n_2, \quad c^* = c_1 - c_2, \quad u^* = u_1 - u_2, \quad P^* = P_1 - P_2.
\end{split}
\end{equation}
Then for $t \in [0,T]$, the triple $(n^*,c^*,u^*)$ satisfies the following system:
\begin{equation}\label{fin2}
\left\{
\begin{aligned}
&n^*_t + u_1 \cdot \nabla n_1 - u_2 \cdot \nabla n_2 = \Delta n^* - \nabla \cdot (n_1 \nabla c_1) + \nabla \cdot (n_2 \nabla c_2), \\
&c^*_t + u_1 \cdot \nabla c_1 - u_2 \cdot \nabla c_2 = \Delta c^* - n_1 c_1 + n_2 c_2, \\
&u^*_t + \textbf{P}[(u_1 \cdot \nabla) u_1 - (u_2 \cdot \nabla) u_2] = -A u^* + \textbf{P}(n^* \nabla \phi), \\
&\nabla \cdot u^* = 0, \\
&(n_0^*, c_0^*, u_0^*) = 0.
\end{aligned}
\right.
\end{equation}
Since $n^* \in L^2(0,T; H^2(\mathbb{R}^2))$ and $\partial_t n^* \in L^2(0,T; L^2(\mathbb{R}^2))$, it follows that $\frac{1}{2} \frac{\mathrm{d}}{\mathrm{d}t} \|n^*\|_{L^2}^2 = (\partial_t n^*, n^*)_{L^2}$. From the first equation of \eqref{fin2}, we obtain
\begin{equation}\label{un-1}
\begin{split}
\frac{1}{2} \frac{\mathrm{d}}{\mathrm{d}t} \|n^*\|_{L^2}^2 + \|\nabla n^*\|_{L^2}^2
&= (u^* n_1, \nabla n^*)_{L^2} + (n^* \nabla c_1, \nabla n^*)_{L^2} + (n_2 \nabla c^*, \nabla n^*)_{L^2} \\
&\leq \frac{1}{2} \|\nabla n^*\|_{L^2}^2 + C \|n_1\|_{H^1}^2 \|u^*\|_{H^1}^2 + C \|c_1\|_{H^3}^2 \|n^*\|_{L^2}^2 \\
&\quad + C \|n_2\|_{H^1}^2 \|c^*\|_{H^2}^2.
\end{split}
\end{equation}
In a similar manner, since $\frac{1}{2} \frac{\mathrm{d}}{\mathrm{d}t} \|\nabla n^*\|_{L^2}^2 = \langle \nabla \partial_t n^*, \nabla n^* \rangle_{H^{-1}, H^1}$, we deduce that
\begin{equation}\label{un-2}
\begin{split}
\frac{1}{2} \frac{\mathrm{d}}{\mathrm{d}t} \|\nabla n^*\|_{L^2}^2 + \|\Delta n^*\|_{L^2}^2
&\leq \frac{1}{2} \|\Delta n^*\|_{L^2}^2 + C \|\nabla n_1\|_{L^4}^2 \|u^*\|_{L^4}^2 + C \|u_2\|_{L^{\infty}}^2 \|\nabla n^*\|_{L^2}^2 \\
&\quad + C \|\nabla c_1\|_{L^{\infty}}^2 \|\nabla n^*\|_{L^2}^2 + C \|\Delta c_1\|_{L^4}^2 \|n^*\|_{L^4}^2 \\
&\quad + C \|\nabla n_2\|_{L^4}^2 \|\nabla c^*\|_{L^4}^2 + C \|n_2\|_{L^{\infty}}^2 \|\Delta c^*\|_{L^2}^2 \\
&\leq \frac{1}{2} \|\Delta n^*\|_{L^2}^2 + C \|n_1\|_{H^2}^2 \|u^*\|_{H^1}^2 + C (\|u_2\|_{H^2}^2 + \|c_1\|_{H^3}^2) \|n^*\|_{H^1}^2 \\
&\quad + C \|n_2\|_{H^2}^2 \|\Delta c^*\|_{L^2}^2.
\end{split}
\end{equation}
For the second equation of \eqref{fin2}, noting that $c^* \in L^2(0,T; H^3(\mathbb{R}^2))$ and $\partial_t c^* \in L^2(0,T; H^1(\mathbb{R}^2))$, it follows that $\frac{1}{2} \frac{\mathrm{d}}{\mathrm{d}t} \|c^*\|_{L^2}^2 = (\partial_t c^*, c^*)_{L^2}$ and $\frac{1}{2} \frac{\mathrm{d}}{\mathrm{d}t} \|\Delta c^*\|_{L^2}^2 = \langle \Delta \partial_t c^*, \Delta c^* \rangle_{H^{-1}, H^1}$. By using the Sobolev embedding, we obtain
\begin{equation}\label{un-3}
\begin{split}
\frac{1}{2} \frac{\mathrm{d}}{\mathrm{d}t} \|c^*\|_{L^2}^2 + \|\nabla c^*\|_{L^2}^2
&= (u^* c_1, \nabla c^*)_{L^2} - (n^* c_1, c^*)_{L^2} - (n_2 c^*, c^*)_{L^2} \\
&\leq \frac{1}{2} \|\nabla c^*\|_{L^2}^2 + C \|c_1\|_{L^{\infty}}^2 \|u^*\|_{L^2}^2 + C \|n^*\|_{L^2}^2 \\
&\quad + C (\|c_1\|_{L^{\infty}}^2 + \|n_2\|_{H^2}^2) \|c^*\|_{L^2}^2.
\end{split}
\end{equation}
\begin{equation}\label{un-4}
\begin{split}
\frac{1}{2} \frac{\mathrm{d}}{\mathrm{d}t} \|\Delta c^*\|_{L^2}^2 + \|\nabla \Delta c^*\|_{L^2}^2
&\leq \frac{1}{2} \|\nabla \Delta c^*\|_{L^2}^2 + C \|c_1\|_{L^{\infty}}^2 \|\nabla n^*\|_{L^2}^2 + C \|c_1\|_{H^3}^2 \|u^*\|_{H^1}^2 \\
&\quad + C \|c_1\|_{H^3}^2 \|n^*\|_{L^2}^2 + C (\|u_2\|_{H^2}^2 + \|n_2\|_{H^1}^2) \|c^*\|_{H^2}^2.
\end{split}
\end{equation}
For the fluid equation, we get
\begin{equation}\label{un-5}
\begin{split}
\frac{1}{2} \frac{\mathrm{d}}{\mathrm{d}t} \|u^*(t)\|_{L^2}^2 + \|\nabla u^*\|_{L^2}^2
&= (u^* \otimes u_1, \nabla u^*)_{L^2} + (n^* \nabla \phi, u^*)_{L^2} \\
&\leq \frac{1}{2} \|\nabla u^*\|_{L^2}^2 + C (1 + \|u_1\|_{H^2}^2) \|u^*\|_{L^2}^2 + C \|\nabla \phi\|_{L^{\infty}}^2 \|n^*\|_{L^2}^2,
\end{split}
\end{equation}
and
\begin{equation}\label{un-6}
\begin{split}
\frac{1}{2} \frac{\mathrm{d}}{\mathrm{d}t} \|\nabla u^*(t)\|_{L^2}^2 + \|\Delta u^*\|_{L^2}^2\leq \frac{1}{2} \|\Delta u^*\|_{L^2}^2 + C \|u_1\|_{H^2}^2 \|u^*\|_{H^1}^2 + C \|\nabla \phi\|_{L^{\infty}}^2 \|n^*\|_{L^2}^2.
\end{split}
\end{equation}
Combining the estimates \eqref{un-1}-\eqref{un-6}, we derive that
\begin{equation}\label{un-7}
\begin{split}
\frac{\mathrm{d}}{\mathrm{d}t} \left( \|n^*(t)\|_{H^1}^2 + \|c^*(t)\|_{H^2}^2 + \|u^*(t)\|_{H^1}^2 \right)
\leq C \mathbf{A}(t) \left( \|n^*(t)\|_{H^1}^2 + \|c^*(t)\|_{H^2}^2 + \|u^*(t)\|_{H^1}^2 \right),
\end{split}
\end{equation}
where
\begin{equation*}
\mathbf{A}(t) := 1 + \|\nabla \phi\|_{L^{\infty}}^2 + \sum_{i=1}^{2} \left( \|n_i(t)\|_{H^2}^2 + \|c_i(t)\|_{H^3}^2 + \|u_i(t)\|_{H^2}^2 \right).
\end{equation*}
According to the time-space regularity provided by \eqref{reg}, we know that for every $t \in [0,T]$,
\begin{equation*}
\int_0^t \mathbf{A}(r)\,\mathrm{d}r < \infty.
\end{equation*}
Thus, applying the Gronwall lemma to \eqref{un-7}, we derive that
\begin{equation*}
\|n^*(t)\|_{H^1}^2 + \|c^*(t)\|_{H^2}^2 + \|u^*(t)\|_{H^1}^2 = 0,
\end{equation*}
which implies the uniqueness. The proof of Theorem \ref{th1} is thus completed. \qed

\section{Proof of Theorem \ref{th2}}\label{sec4}
This section is dedicated to providing the core proof process for Theorem \ref{th2}. Under the assumption that the initial data $(n_0, c_0, u_0)$ and the potential function $\phi$ for system \eqref{CNS0} possess higher regularity as specified in Theorem \ref{th2}, we can, building upon Lemma \ref{5pro}, further derive uniform higher-order estimates via an iterative procedure. The precise result is stated as follows.

\begin{proposition} \label{6pro}
Let the assumptions in Theorem~\ref{th2} hold, and let $(n^\epsilon, c^\epsilon, u^\epsilon)$ be the global solution to \eqref{Mod-1} with $\epsilon \in (0,1)$. Then for every integer $s \geq 2$, we have
\begin{equation} \begin{split} \label{6pro-1}
&\|n^{\epsilon}(t)\|_{H^s}^2 + \|c^{\epsilon}(t)\|_{H^{s+1}}^2 + \|u^{\epsilon}(t)\|_{H^s}^2 \\
&\quad + \int_0^t \left( \|n^{\epsilon}(r)\|_{H^{s+1}}^2 + \|c^{\epsilon}(r)\|_{H^{s+2}}^2 + \|u^{\epsilon}(r)\|_{H^{s+1}}^2 \right) \,\mathrm{d}r \lesssim_{n_0, c_0, u_0, \phi, t} 1.
\end{split} \end{equation}
\end{proposition}

\begin{proof}[\textbf{Proof}]
Applying the Bessel potential operator $\Lambda^s$ to the third equation in \eqref{Mod-1} and taking the $L^2$-inner product with $\Lambda^s u^\epsilon$, we deduce via the Moser-type inequality and the Sobolev embedding $H^2(\mathbb{R}^2)\hookrightarrow L^\infty(\mathbb{R}^2)$ that
\begin{equation} \label{4lem-2}
\begin{split}
\frac{1}{2} \frac{\mathrm{d}}{\mathrm{d}t} \|\Lambda^s u^\epsilon\|_{L^2}^2 + \|\nabla \Lambda^s u^\epsilon\|_{L^2}^2
&= (\Lambda^s(u^\epsilon \otimes u^\epsilon), \nabla \Lambda^s u^\epsilon) + (\Lambda^s((n^\epsilon \nabla \phi) * \rho^\epsilon), \Lambda^s u^\epsilon) \\
&\leq \frac{1}{4} \|\nabla \Lambda^s u^\epsilon\|_{L^2}^2 + \frac{1}{4} \|\Lambda^s u^\epsilon\|_{L^2}^2 \\
&\quad + C_s \|u^\epsilon\|_{H^2}^2 \|\Lambda^s u^\epsilon\|_{L^2}^2 + C_s \|\phi\|_{H^{s+1}}^2 \|n^\epsilon\|_{H^s}^2.
\end{split}
\end{equation}
Similarly, for the first equation, we obtain
\begin{equation} \label{4lem-3}
\begin{split}
\frac{\mathrm{d}}{\mathrm{d}t} \|n^\epsilon\|_{H^s}^2 + \|\nabla n^\epsilon\|_{H^s}^2
&\lesssim \|u^\epsilon n^\epsilon\|_{H^s}^2 + \|n^\epsilon \nabla (c^\epsilon * \rho^\epsilon)\|_{H^s}^2 \\
&\lesssim_s \|n^\epsilon\|_{H^2}^2 \|u^\epsilon\|_{H^s}^2 + \|n^\epsilon\|_{H^2}^2 \|c^\epsilon\|_{H^{s+1}}^2 \\
&\quad + (\|u^\epsilon\|_{H^2}^2 + \|c^\epsilon\|_{H^3}^2) \|n^\epsilon\|_{H^s}^2.
\end{split}
\end{equation}
For the second equation, note that $\Lambda$ is a self-adjoint operator on $L^2(\mathbb{R}^2)$, and
$
\|\Lambda^{s+2}f\|_{L^2}^2 \asymp\|\Lambda^{s+1}f\|_{L^2}^2 + \|\nabla \Lambda^{s+1}f\|_{L^2}^2,
$
we proceed analogously:
\begin{equation} \label{4lem-4}
\begin{split}
\frac{\mathrm{d}}{\mathrm{d}t} \|c^\epsilon\|_{H^{s+1}}^2 + \|\nabla c^\epsilon\|_{H^{s+1}}^2
&=-(\Lambda^{s}(u^\epsilon \cdot \nabla c^\epsilon), \Lambda^{s+2} c^\epsilon) - (\Lambda^{s}(c^\epsilon (n^\epsilon * \rho^\epsilon)),\Lambda^{s+2} c^\epsilon) \\
&\leq \frac{1}{2} \|\nabla c^\epsilon\|_{H^{s+1}}^2 + C_s \|c^\epsilon\|_{H^3}^2 \|u^\epsilon\|_{H^s}^2 \\
&\quad + C_s \|c^\epsilon\|_{H^2}^2 \|n^\epsilon\|_{H^s}^2 + C_s (1+\|u^\epsilon\|_{H^2}^2 + \|n^\epsilon\|_{H^2}^2) \|c^\epsilon\|_{H^{s+1}}^2.
\end{split}
\end{equation}
Combining estimates \eqref{4lem-2}--\eqref{4lem-4}, we find
\begin{equation} \label{4lem-5}
\begin{split}
\frac{\mathrm{d}}{\mathrm{d}t} \mathcal{E}(t) + \|\nabla n^\epsilon(t)\|_{H^s}^2 + \|\nabla c^\epsilon(t)\|_{H^{s+1}}^2 + \|\nabla u^\epsilon(t)\|_{H^s}^2
\lesssim_s \mathbf{A}_1(t) \cdot \mathcal{E}(t),
\end{split}
\end{equation}
where
$$
\mathcal{E}(t) := \|n^\epsilon(t)\|_{H^s}^2 + \|c^\epsilon(t)\|_{H^{s+1}}^2 + \|u^\epsilon(t)\|_{H^s}^2,
$$
and
$$
\mathbf{A}_1(t) := 1+\|n^\epsilon(t)\|_{H^2}^2 + \|c^\epsilon(t)\|_{H^3}^2 + \|u^\epsilon(t)\|_{H^2}^2 + \|\phi\|_{H^{s+1}}^2.
$$
Thanks to the uniform-in-\(\epsilon\) estimates established in \eqref{5pro-1}-\eqref{5pro-3}, we have
$$
\int_0^t \mathbf{A}_1(r)\,\mathrm{d}r \lesssim_{n_0, c_0, u_0, \phi, t} 1.
$$
Thus, by Gronwall's inequality, we deduce that \eqref{6pro-1} holds for every fixed integer \(s \geq 2\). This completes the proof of Proposition~\ref{6pro}.
\end{proof}

\begin{corollary} \label{cor2}
Let the assumptions in Theorem~\ref{th2} be satisfied. Then the system \eqref{CNS0} admits at least one solution
\begin{equation}\label{regA}
\begin{split}
&n \in \mathcal{C}([0,T]; H^s(\mathbb{R}^2)) \cap L^2(0,T; H^{s+1}(\mathbb{R}^2)),\quad
c \in \mathcal{C}([0,T]; H^{s+1}(\mathbb{R}^2)) \cap L^2(0,T; H^{s+2}(\mathbb{R}^2)),\\
&u \in \mathcal{C}([0,T]; \mathbb{H}^s) \cap L^2(0,T; \mathbb{H}^{s+1}).
\end{split}
\end{equation}
In particular, if $s = 4$, then system \eqref{CNS0} possesses a global classical solution
\begin{equation}\label{regB}
\begin{split}
&n \in \mathcal{C}([0,T]; \mathcal{C}^2(\mathbb{R}^2)) \cap \mathcal{C}^1([0,T]; \mathcal{C}(\mathbb{R}^2)),\quad
c \in \mathcal{C}([0,T]; \mathcal{C}^3(\mathbb{R}^2)) \cap \mathcal{C}^1([0,T]; \mathcal{C}^2(\mathbb{R}^2)),\\
&u \in \mathcal{C}([0,T]; \mathcal{C}^2(\mathbb{R}^2;\mathbb{R}^2)) \cap \mathcal{C}^1([0,T]; \mathcal{C}(\mathbb{R}^2;\mathbb{R}^2)).
\end{split}
\end{equation}
Moreover, if $s = 4 + 2m$ for some $m \in \mathbb{N}^+$, then system \eqref{CNS0} admits a global smooth solution
\begin{equation}\label{regC}
\begin{split}
&n \in \bigcap_{k=0}^{m+1} \mathcal{C}^k([0,T]; \mathcal{C}^{2(m+1-k)}(\mathbb{R}^2)),\quad
c \in \bigcap_{k=0}^{m+1} \mathcal{C}^k([0,T]; \mathcal{C}^{2(m+1-k)+1}(\mathbb{R}^2)),\\
&u \in \bigcap_{k=0}^{m+1} \mathcal{C}^k([0,T]; \mathcal{C}^{2(m+1-k)}(\mathbb{R}^2; \mathbb{R}^2)).
\end{split}
\end{equation}
\end{corollary}

\begin{proof}[\emph{\textbf{Proof}}]
Based on the key uniform estimates provided by Proposition~\ref{6pro}, it is not difficult to prove that for every integer $s \geq 2$,
\begin{equation}\label{0a}
\begin{split}
\left\|\frac{\partial n^{\epsilon}}{\partial t}\right\|_{L^2(0,T;H^{s-1})}^2
+ \left\|\frac{\partial c^{\epsilon}}{\partial t}\right\|_{L^2(0,T;H^{s})}^2
+ \left\|\frac{\partial u^{\epsilon}}{\partial t}\right\|_{L^2(0,T;H^{s-1})}^2
\lesssim_{n_0, c_0, u_0, \phi, T} 1.
\end{split}
\end{equation}
Similar to the corresponding discussion in Subsection~\ref{sec3-1}, it is standard to conclude that $(n,c,u)$ satisfies the integral identities~\eqref{iden}. Moreover, by using the Lions-Magenes lemma, it follows from \eqref{6pro-1} and \eqref{0a} that the solution $(n,c,u)$ satisfies the time-space regularity described in \eqref{regA}.

In particular, the integral identities in \eqref{iden} imply that
\begin{equation}\label{1a}
\begin{split}
n(t,x) &= n_0(x) + \int_0^t [\Delta n(r,x) - u(r,x)\cdot\nabla n(r,x) - \nabla\cdot (n(r,x)\nabla c(r,x))]\mathrm{d}r,\\
c(t,x) &= c_0(x) + \int_0^t [\Delta c(r,x) - u(r,x)\cdot\nabla c(r,x) - n(r,x)c(r,x)]\mathrm{d}r,\\
u(t,x) &= u_0(x) + \int_0^t [\Delta u(r,x) - (u(r,x)\cdot\nabla)u(r,x) + \nabla P(t,x) + n(r,x)\nabla\phi(x)]\mathrm{d}r
\end{split}
\end{equation}
hold for almost every $x \in \mathbb{R}^2$ and all $t \in [0,T]$.

When $s = 4$, since $H^2(\mathbb{R}^2)$ is a Banach algebra, it follows from \eqref{regA} that
\begin{equation}\label{2a}
\begin{split}
&\Delta n - u \cdot \nabla n - \nabla \cdot (n \nabla c) \in \mathcal{C}([0,T]; H^2(\mathbb{R}^2)) \subset \mathcal{C}([0,T]; \mathcal{C}(\mathbb{R}^2)),\\
&\Delta c - u \cdot \nabla c - nc \in \mathcal{C}([0,T]; H^3(\mathbb{R}^2)) \subset \mathcal{C}([0,T]; \mathcal{C}^1(\mathbb{R}^2)),\\
&\Delta u - (u \cdot \nabla)u + \nabla P + n \nabla \phi \in \mathcal{C}([0,T]; H^2(\mathbb{R}^2; \mathbb{R}^2)) \subset \mathcal{C}([0,T]; \mathcal{C}(\mathbb{R}^2; \mathbb{R}^2)).
\end{split}
\end{equation}
Moreover, since two continuous functions that are equal almost everywhere in $\mathbb{R}^2$ must be identical on the whole space, it follows from \eqref{2a} that $(n,c,u)$, possessing the time-space regularity \eqref{regB}, constitutes a global classical solution to system~\eqref{CNS0} when $s = 4$.

Finally, taking advantage of the intrinsic parabolic character of the CNS system~\eqref{CNS0} and applying a standard bootstrap argument, we obtain that the solution $(n,c,u)$ possesses the enhanced regularity~\eqref{regC} for all $s = 4 + 2m$, with $m \in \mathbb{N}^+$. The proof is thus completed.
\end{proof}

Next, we will demonstrate the uniqueness of the solution to complete the proof of Theorem \ref{th2}.

\begin{proposition} \label{Pro}
Let the assumptions in Theorem \ref{th2} be satisfied. Then the global solution $(n,c,u)$ provided in Corollary \ref{cor2} is unique.
\end{proposition}

\begin{proof}[\emph{\textbf{Proof}}]
For convenience, we continue to employ the notations introduced in \eqref{fin1} and \eqref{fin2}. By directly computing $\|\Lambda^s n^*\|_{L^2}^2$, we obtain
\begin{equation} \begin{split}\label{Pro1}
&\frac{1}{2}\frac{\mathrm{d}}{\mathrm{d}t}\|n^*(t)\|_{H^s}^2+\|\nabla n^*\|_{H^s}^2\\
&=(\Lambda^s (u^*n_1),\nabla\Lambda^sn^*)_{L^2}+(\Lambda^s(u_2n^*),\nabla\Lambda^sn^*)_{L^2}+(\Lambda^s(n^*\nabla c_1),\nabla\Lambda^s n^*)_{L^2}\\
&\quad+(\Lambda^s(n_2\nabla c^*),\nabla\Lambda^s n^*)_{L^2}\\
&\leq\frac{1}{2}\|\nabla n^*\|_{H^s}^2+C\|u^*n_1\|_{H^s}^2+C\|u_2n^*\|_{H^s}^2+C\|n^*\nabla c_1\|_{H^s}^2+C\|n_2\nabla c^*\|_{H^s}^2\\
&\leq\frac{1}{2}\|\nabla n^*\|_{H^s}^2+C(\|u_2\|_{H^s}^2+\|c_1\|_{H^{s+1}}^2)\|n^*\|_{H^s}^2+C\|n_2\|_{H^s}^2\|c^*\|_{H^{s+1}}^2+C\|n_1\|_{H^s}^2\|u^*\|_{H^s}^2.
\end{split} \end{equation}
Similarly, we have
\begin{equation} \begin{split}\label{Pro2}
&\frac{1}{2}\frac{\mathrm{d}}{\mathrm{d}t}\|c^*(t)\|_{H^{s+1}}^2+\|\nabla c^*\|_{H^{s+1}}^2\\
&\leq\frac{1}{2}\|\nabla c^*\|_{H^{s+1}}^2+C\|c_1\|_{H^{s+1}}^2\|n^*\|_{H^{s}}^2+C(1+\|u_2\|_{H^{s+1}}^2+\|n_2\|_{H^{s+1}}^2)\|c^*\|_{H^{s+1}}^2\\
&\quad+C\|c_1\|_{H^{s+1}}^2\|u^*\|_{H^{s}}^2,
\end{split} \end{equation}
and
\begin{equation} \begin{split}\label{Pro3}
&\frac{1}{2}\frac{\mathrm{d}}{\mathrm{d}t}\|u^*(t)\|_{H^s}^2+\|\nabla u^*\|_{H^s}^2\\
&\leq\frac{1}{2}\|\nabla u^*\|_{H^s}^2+C(\|u_1\|_{H^s}^2+\|u_2\|_{H^s}^2)\|u^*\|_{H^s}^2+\|\phi\|_{H^{s+1}}^2\|n^*\|_{H^s}^2.
\end{split} \end{equation}
Combining estimates \eqref{Pro1}, \eqref{Pro2}, and \eqref{Pro3}, we derive
\begin{equation} \begin{split}\label{Pro6}
&\frac{\mathrm{d}}{\mathrm{d}t}(\|n^*(t)\|_{H^s}^2+\|c^*(t)\|_{H^{s+1}}^2+\|u^*(t)\|_{H^s}^2)\lesssim\mathbf{B}(t)(\|n^*(t)\|_{H^s}^2+\|c^*(t)\|_{H^{s+1}}^2+\|u^*(t)\|_{H^s}^2),
\end{split} \end{equation}
where
\begin{equation*} \begin{split}
&\mathbf{B}(t):=1 + \|\phi\|_{H^{s+1}}^2 + \sum_{i=1}^{2} \left( \|n_i(t)\|_{H^s}^2 + \|c_i(t)\|_{H^{s+1}}^2 + \|u_i(t)\|_{H^s}^2 \right).
\end{split} \end{equation*}
Therefore, by Gronwall's inequality and the integrability of $\mathbf{B}(t)$ over $(0,t)$, i.e., $\int_0^t\mathbf{B}(r)\mathrm{d}r<\infty$, it follows that
\begin{equation*} \begin{split}
\|n^*(t)\|_{H^s}^2+\|c^*(t)\|_{H^{s+1}}^2+\|u^*(t)\|_{H^s}^2=0,
\end{split} \end{equation*}
which completes the proof of uniqueness.
\end{proof}

The existence result established in Corollary \ref{cor2}, together with the uniqueness provided by Proposition \ref{Pro}, yields the conclusions of Theorem \ref{th2}.\qed

\section*{Acknowledgments}
This work was supported by the National Natural Science Foundation of China (Grant No. 12231008). The author would like to express sincere thanks to Feng Dai, Bo Yang, and Lei Zhang for their helpful discussions and valuable comments.

\section*{Conflict of interest statement}
The author declares that there is no conflict of interest regarding this work.

\section*{Data availability}
No data was used for the research described in this article.

\bibliographystyle{amsrefs}
\bibliography{2DCNS}

\end{document}